\documentclass[a4paper,11pt,twoside]{article}
\usepackage{latexsym}
\usepackage{a4wide}
\usepackage{amscd}
\usepackage{graphics}
\usepackage{graphicx}
\usepackage[english]{varioref}
\usepackage{amsmath}
\usepackage{amssymb}
\usepackage{mathrsfs}
\usepackage{amsthm}
\usepackage{amssymb,tikz}
\usepackage{verbatim}

\usepackage{bbm}
\usepackage{soul,color}
\usepackage{accents}
\usepackage{enumerate}
\usepackage[utf8x]{inputenc}
\usepackage[T1]{fontenc}
\usepackage[english]{babel}
\usepackage{amsfonts, amscd, amsmath, amssymb, amsthm, mathrsfs, mathtools}
\usepackage{braket}
\input xy
\xyoption{all}
\usepackage[utf8x]{inputenc}
\usepackage[T1]{fontenc}
\usepackage[english]{babel}
\usepackage{amsfonts, amscd, amsmath, amssymb, amsthm, mathrsfs, mathtools}
\usepackage{graphicx}
\usepackage{subfig}
\usepackage{braket}
\usepackage[left=2.5cm,right=2.5cm,top=2.5cm,bottom=2.5cm]{geometry}
\usepackage{fancyhdr}
\theoremstyle{definition}

\newtheorem{remark}{Remark}[section]

\newtheorem{esempio}{Example}[section]
\newtheorem{notazione}{Notation}[section]
\newtheorem{assumption}{Assumption}[section]
\newtheorem{claim}{Claim}[section]

\theoremstyle{plain}
\newtheorem{definizione}{Definition}[section]
\newtheorem{teorema}{Theorem}[section]
\newtheorem{proposizione}{Proposition}[section]
\newtheorem{lemma}{Lemma}[section]
\newtheorem{corollario}{Corollary}[section]

\newcommand{\numberset}{\mathbb}
\newcommand{\C}{\numberset{C}}
\newcommand{\R}{\numberset{R}}
\newcommand{\N}{\numberset{N}}
\newcommand{\Z}{\numberset{Z}}

\usepackage{bookmark,hyperref}

\DeclarePairedDelimiter{\abs}{\lvert}{\rvert}
\DeclarePairedDelimiter{\norma}{\lVert}{\rVert}

\makeatletter
\let\oldabs\abs
\def\abs{\@ifstar{\oldabs}{\oldabs*}}
\let\oldnorma\norma
\def\norma{\@ifstar{\oldnorma}{\oldnorma*}}

\title{Torsion and Linking number for a surface diffeomorphism}

\author{\hspace{-1cm}{ ANNA FLORIO${\,
}^{1}$}
\vspace{0.3cm}
\\
\hspace{-1cm}${\ }^{1}$ Laboratoire de Mathématiques d'Avignon, Avignon Université, \\
\hspace{-1cm} 84018 Avignon, France\\}

\date{ }

\begin{document}

\maketitle

\begin{abstract}
\noindent For a ${\cal{C}}^1$ diffeomorphism $f:\R^2\rightarrow\R^2$ isotopic to the identity, we prove that for any value $l\in\R$ of the linking number at finite time of the orbits of two points there exists at least a point whose torsion at the same finite time equals $l\in\R$. As an outcome, we give a much simplier proof of a theorem by Matsumoto and Nakayama concerning torsion of measures on $\mathbb{T}^2$. In addition, in the framework of twist maps, we generalize a known result concerning the linking number of periodic points: indeed, we estimate such value for any couple of points for which the limit of the linking number exists. 
\end{abstract}

\section{Introduction}
Let $S$ be a parallelizable Riemannian surface and let $(f_t)_{t\in[0,1]}$ be an isotopy in $\text{Diff}^1(S)$ joining the identity $Id_S$ to $f_1=f$. The tangent bundle inherits a dynamics through the differential $Df_t:TS\rightarrow TS$.\\ 
\noindent We are interested in the \emph{rotation number} which was first introduced by D. Ruelle in \cite{ruelle}. Roughly speaking, Ruelle's rotation number (which we will call \emph{torsion} from now on) refers to the asymptotic velocity at which the tangent-to-$S$ vectors turn. 
\noindent If $\xi\in T_x S\setminus\{0\}$, then $Torsion_T((f_t)_t,x,\xi)$ is the variation between $0$ and $T$ of a continuous determination of the angle function associated to $Df_t(x)\xi$, $t\in[0,T]$, divided by $T$.
\noindent The \emph{torsion at the orbit of $x$} is the limit for $T\rightarrow +\infty$ of $Torsion_T((f_t)_t,x,\xi)$, whenever it exists. Hence, when defined, it represents the asymptotic rotational behavior of vectors in the tangent space of $x$. Ruelle proved that for any $f$-invariant probability measure with compact support, for almost every point, the torsion exists for all vectors $\xi\in T_x S\setminus\{0\}$. \\
\indent Besides, in the setting of the plane $\R^2$, for any $x,y\in\R^2$, $x\neq y$, the linking number of the orbits of $x$ and $y$ is the asymptotic angular velocity of the vector $f_t(y)-f_t(x)$. We first introduce $Linking_T((f_t)_t,x,y)$ as the variation between $0$ and $T$ of a continuous determination of the angle associated to $f_t(y)-f_t(x)$, $t\in[0,T]$, divided by the length of the time interval.\\
\noindent Similarly to what done for the torsion, the linking number of a pair of orbits is, whenever it exists, the limit for $T\rightarrow +\infty$ of the previous quantity. For any $f$-invariant probability measure with compact support, for almost every pair of points, the linking number exists (see Remark \ref{remark esistenza linking misura}).\\

\indent Besides Ruelle's article, J. Mather in \cite{matherexistence}, \cite{matheramount} and S. Angenent in \cite{angenent} have particularly worked on null torsion sets through a more variational point of view. Also S. Crovisier in \cite{crovisier} worked in the same setting with a topological approach. F. B\'eguin and Z. Boubaker in \cite{beguin} gave conditions which assure the existence of orbits with non zero torsion and used the concept of linking number presented above.\\
In their article, Beguin and Boubaker raise the following question: for a given isotopy $(f_t)_t:[0,1]\rightarrow \text{Diff}^1(\R^2)$ starting at the identity, assuming that the linking number of two points $x,y$ is non-zero, does there exist at least a point $z$ on the segment connecting $x$ and $y$ such that its torsion is also non-zero? Their answer is positive but we give here a more precise result.\\
\noindent We state the following
\begin{teorema}\label{thm intro 1}
Let $(F_t)_{t\in[0,1]}$ be an isotopy in $\text{Diff}^1(\R^2)$ joining $Id_{\R^2}$ to $F_1=F$. Assume that there exist two points $x,y\in\R^2$, $x\neq y$ such that
$$
Linking_1((F_t)_t,x,y)=l\in\R.
$$
Then there exists a point $z\in[x,y]$ so that
$$
Torsion_1\left((F_t)_t,z,y-x\right)=l.
$$
\end{teorema}
\noindent At first sight, this result could recall a mean value theorem but the arguments and the strategies needed in the proof (which composes most of the article) are much more sophisticated and subtle.\\
\noindent First of all, through continuous modifications of the isotopy, we bring ourselves in a rotation frame of reference, reducing then the discussion to the case $l=0$.\\
\noindent In order to avoid self-intersections of the curve, a passage to the universal covering of the punctured plane is required: the strategy in doing so is using a polar coordinate frame, with respect to which one of the endpoints of the segment coincides with the singularity.\\
\noindent Finally, we carefully study the behavior of points in a neighborhood of the singularity $x$, the point previously blown up which corresponds to the origin of the polar coordinate framework. We then apply the Turning Tangent Theorem (see \cite{docarmo}, Chapter 4, Section 5).\\
\indent Passing to asymptotic quantities, we then prove the existence of $f$-invariant Borel probability measures $\mu$ whose torsion, i.e. $\int_S Torsion((f_t)_t,x)d\mu(x)$, equals $l\in\R$, where now $l$ is the asymptotic linking number of the orbits of two points.\\
\noindent As an outcome, for a ${\cal{C}}^1$ diffeomorphism $f$ over $\mathbb{T}^2$, we prove the existence of a $f$-invariant Borel probability measure of null torsion on the torus.
\noindent This result has already been proved by S. Matsumoto and H. Nakayama in \cite{matsumoto}, but our setting provides a much easier proof and requires only a ${\cal{C}}^1$ regularity of $f$.\\

\indent We then focus our attention on dynamical systems of specific interest and we analyse (positive) twist maps. The framework is the annulus $\mathbb{A}=\mathbb{T}\times\R$ on which we consider an isotopy $(f_t)_{t\in[0,1]}$ in $\text{Diff}^1(\mathbb{A})$ joining $Id_{\mathbb{A}}$ to $f_1=f$ and we assume that $f:\mathbb{A}\rightarrow\mathbb{A}$ is a ${\cal{C}}^1$ positive twist map. That is to say, for any lift $F:\R^2\rightarrow\R^2$ of the map $f$ and for any $x\in\R$ the function
$$
y\mapsto p_1\circ F(x,y)
$$
is an increasing diffeomorphism of $\R$, where $p_1:\R^2\rightarrow\R$ denotes the first coordinate projection.\\
\indent The interest for twist maps has largely spread all along the literature (see \cite{lecalvezproprietes}, \cite{matherexistence}, \cite{mathervariational}, \cite{matherglancing} and \cite{moser}) and several authors have studied their connection with the notion of torsion.\\
\noindent The intuitive geometric idea that the deviation property must oblige the vertical vector of each tangent space turning to the right is translated into a negative torsion at finite time with respect to the vector $(0,1)$ (according to a counterclockwise orientation of the tangent plane). This is already remarked in other works, such as \cite{crovisier} and \cite{lecalvezproprietes}. Anyway, we give a better estimation both for torsion and for linking number.
\begin{teorema}\label{thm intro 2}
Let $f:\mathbb{A}\rightarrow\mathbb{A}$ be a ${\cal{C}}^1$ positive twist map and let $(f_t)_{t\in[0,1]}$ be an isotopy in $\text{Diff}^1(\mathbb{A})$ joining the identity to $f$. Let $(0,1)\in T_{z}\mathbb{A}\setminus\{0\}$ be the positive vertical unitary tangent vector which lies in the tangent space of every point $z\in\mathbb{A}$. Then, for any $z\in\mathbb{A}$ and for any $n\in\N,n\neq 0$, it holds
$$
Torsion_n((f_t)_t,z,(0,1))\in(-\pi,0).
$$
\end{teorema}
\noindent As a consequence of this first bound, it follows
\begin{corollario}\label{cor intro 2}
Let $f:\mathbb{A}\rightarrow\mathbb{A}$ be a ${\cal{C}}^1$ positive twist map and let $(f_t)_{t\in[0,1]}$ be an isotopy in $\text{Diff}^1(\mathbb{A})$ joining the identity to $f_1=f$. Then, for any point $z\in\mathbb{A}$ for which the limit of the torsion exists, it holds
$$
Torsion((f_t)_t,z)\in[-\pi,0].
$$
\end{corollario}
\noindent The proof of this result first requires to show the independence of the torsion from the choice of the isotopy for a ${\cal{C}}^1$ diffeomorphism on $\mathbb{A}$ (not necessarily a twist map), which is actually an outcome of Theorem \ref{thm intro 1}.\\
\noindent According to what done by P. Le Calvez in \cite{lecalvezproprietes}, it is known that a (positive) twist map can be joined to $Id_{\mathbb{A}}$ through an isotopy such that $f_t$ is still a (positive) twist map for any $t\in(0,1]$. This condition assures us that the image of the vertical vector $(0,1)$ through $Df_t$ is strictly contained in the right half-plane and therefore it cannot cross the vertical again. Through a non trivial induction argument, we deduce Theorem \ref{thm intro 2}.\\
\indent Thanks to Theorem \ref{thm intro 1}, the same estimation holds true also for the linking number of the orbits of any two points in the lifted framework. Indeed:
\begin{corollario}\label{thm intro 3}
Let $F:\R^2\rightarrow\R^2$ be a lift of a ${\cal{C}}^1$ positive twist map on $\mathbb{A}$. Let $(F_t)_t$ be the isotopy joining the identity to $F$, obtained as a lift of an isotopy on $\mathbb{A}$ joining $Id_{\mathbb{A}}$ to the twist map. Let $z_1,z_2\in\R^2$, $z_1\neq z_2$ be such that their linking number exists. Then
$$
Linking((F_t),z_1,z_2)\in[-\pi,0].
$$
\end{corollario}
\noindent This result was already known by P. Le Calvez for periodic orbits and then, through the ${\cal{C}}^0$ Closing Lemma, also for $F$-invariant measures, but our theorem generalizes it, holding true for any couples of points for which the (asymptotic) linking number exists.\\

\indent The paper is organized as follows. In Section 2 the notation is fixed and the main definitions concerning torsion and linking number are provided. Moreover, we present some useful properties and characterisations of these notions (see Subsections \ref{subsection proprieta 1} and \ref{subsection proprieta 2}). Section 3 is devoted to prove Theorem \ref{thm intro 1} in all details. Some consequences are also presented. Finally, Section 4 focuses on the main results on torsion and linking number for positive twist maps, i.e. Theorem \ref{thm intro 2} and Corollaries \ref{thm intro 3} and \ref{cor intro 2}, remarking also some other interesting implications of them. In addition, we present some examples in which the torsion and the linking number reach the extremal values $0$ and $-\pi$.\\
~\newline
\noindent\textbf{Acknowledgements.} \noindent The author is extremely grateful to Professor Marie-Claude Arnaud and Andrea Venturelli for their precious advices to improve the text and for many stimulating discussions. The author acknowledges the anonymous referee for his or her useful remarks and observations.

\section{Definitions and first properties}
\subsection{Notation}\label{Notation}
Throughout the work we assume $\R^2$ endowed with the canonical euclidean metric and the standard trivialization. Denote as $\mathbb{T}$ the quotient space $\R/2\pi\Z$ and as
$$
p:\R\rightarrow\mathbb{T}
$$
$$
x\mapsto \bar{x}=x\ mod\  2\pi
$$
the universal covering of the 1-dimensional torus $\mathbb{T}$.\\
\noindent We use the notation $\mathbb{A}$ for the product space $\mathbb{T}\times\R$ and 
$$
P:\R^2\rightarrow\mathbb{A}
$$
$$
(x,y)\mapsto (\bar{x},y)=(x\ mod\  2\pi, y)
$$
for the universal covering of the annulus $\mathbb{A}$. A point of the annulus is denoted by $\bar{z}=(\bar{x},y)\in\mathbb{A}$, while $z=(x,y)\in\R^2$ refers to a lift of $\bar{z}$ over $\R^2$.\\
\noindent The functions
\begin{equation}
\bar{p}_1:\mathbb{A}\rightarrow\mathbb{T},\quad (\bar{x},y)\mapsto \bar{x}
\end{equation}
\begin{equation}
\bar{p}_2:\mathbb{A}\rightarrow\R,\quad (\bar{x},y)\mapsto y
\end{equation}
are the projections over the first and the second coordinates, respectively; the coordinate projections of $\R^2$ are denoted as $p_1,p_2$.\\
\noindent The $2$-dimensional torus is the quotient space
$$
\mathbb{T}^2 := \R^2/(2\pi\Z)^2.
$$

\noindent All along the work, the counterclockwise orientation of the plane is chosen.\\
\noindent Once provided a Riemannian metric\footnote{Recall that in this work we are endowing $\R^2$ with the standard Riemannian metric.} and an orientation, the \emph{oriented angle} between two non-zero vectors $u,v\in\R^2$ is well defined as an element of $\mathbb{T}$. A \emph{measure of the angle} is an element of $\R$ whose image through $p$ coincides with the oriented angle.

\noindent The notation ${\cal{R}}(a,\psi)$ refers to the rotation of the plane $\R^2$ of center $a\in\R^2$ and angle $\psi$, while $\tau_v$ denotes the translation on the plane by the vector $v\in\R^2$.

\noindent A fundamental notion will be that of \emph{isotopy}:

\begin{definizione}\label{isotopia}
Let $M,N$ be differential manifolds and let $f,g:M\rightarrow N$ be in $\text{Diff }^1(M,N)$.

\noindent An \emph{isotopy} $(\psi_t)_{t\in[0,1]}$ joining $f$ to $g$ is an arc in $\text{Diff }^1(M,N)$
such that $\psi_0=f,\psi_1=g$ and which is continuous with respect to the \emph{weak} or \emph{compact-open} ${\cal{C}}^1$ topology on $\text{Diff }^1(M,N)$.
\end{definizione}

\begin{definizione}
Let $I\subseteq\R$ be an interval. A \emph{continuous determination} of an angle function $\theta:I\rightarrow\mathbb{T}$ is a continuous lift of $\theta$, i.e. a continuous function $\tilde{\theta}:I\rightarrow\R$ such that $\tilde{\theta}(s)$ is a measure of the oriented angle $\theta(s)$ for any $s\in I$.
\end{definizione}
\noindent We remark that a necessary and sufficient condition for the existence of a continuous determination is the continuity of its angle function.

\subsection{Definition of Torsion and Linking number}
\noindent Let $S$ be a parallelizable surface. Denote as $TS_*$ the set $\{ (x,\xi) :\ x\in S,\ \xi\in T_xS\setminus\{0\} \}$. The notation $T^1S$ refers to the unitary tangent bundle. We fix an orientation and we endow $S$ with a Riemannian metric: the notion of oriented angle between two non zero vectors of the same tangent space is well-defined.\\
\begin{remark}\label{remark vectorfield}
The choice of an orientation and of a reference continuous vector field $X$ over $S$ that never vanishes is equivalent to that of a trivialization diffeomorphism. Indeed, on every tangent space we define the endomorphism $J$
$$
J:T_xS\rightarrow T_xS
$$
as a rotation of angle $\frac{\pi}{2}$ according to the fixed orientation. It holds that $J^2=-Id$.\\
\noindent For any $x\in S$, $(X(x),JX(x))$ provides a direct basis of the tangent space. A trivialization diffeomorphism is so given by
$$
TS\ni (x;\alpha X(x)+\beta JX(x))\overset{\phi}{\mapsto}(x;\alpha,\beta)\in S\times\R^2
$$
where $\alpha,\beta\in\R$ are the coordinates with respect to the basis $(X(x),JX(x))$.
\end{remark}
\noindent Let $(f_t)_{t\in[0,1]}$ be an isotopy joining the identity to $f_1=f$. We then extend the isotopy for any positive time in the following way: let $t\in\R_+$, then the ${\cal{C}}^1$ diffeomorphism $f_t:S\rightarrow S$ is defined as
$$
f_t := f_{\{t\}}\circ f^{\lfloor t\rfloor}
$$
where $\{\cdot\},\lfloor\cdot\rfloor$ denote the fractionary and integer part of $t$, respectively.
\begin{notazione}
With an abuse of notation, we also denote the extended isotopy as $(f_t)_t$.\\
\noindent In addition, we fix a reference vector field $X$ that never vanishes (see Remark \ref{remark vectorfield}). Suppose that $X(x)$ has unitary norm for any $x\in S$. We will make explicit the choice of $X$ when needed. We recall the notation $\theta(u,v)$ for the oriented angle between two non zero vectors $v$ and $u$.
\end{notazione}

\noindent Our definition of \emph{torsion}, the one given by B\'eguin and Boubaker in \cite{beguin}, actually coincides with Ruelle's notion of \emph{rotation number} (see \cite{ruelle}).\\
\begin{definizione}\label{definizione v tilde aa}
Let $S$ be a parallelizable surface and let $(f_t)_{t\in[0,1]}$ be an isotopy in $\text{Diff}^1(S)$ joining the identity $Id_S$ to $f_1=f$. Then, we define the function $v((f_t)_t)$ as follows:
\begin{equation}\label{definizione funzione v torsion}
\begin{split}
v((f_t)_t):&TS_*\times\R\rightarrow\mathbb{T}\\
&(x,\xi,t)\mapsto \theta\left( X(f_t(x)), Df_t(x)\xi \right).
\end{split}
\end{equation}
\noindent Fix then $(x,\xi)\in TS_*$ and, since the angle function $v((f_t)_t)(x,\xi,\cdot)$ is continuous, consider a continuous determination $\tilde{v}((f_t)_t)(x,\xi,\cdot):\R\rightarrow\R$ of it.\\
\end{definizione}
\begin{definizione}
Let $S$ and $f$ be as above. Let $x\in S$ and $\xi\in T_xS\setminus\{0\}$. Consider $v((f_t)_t)(x,\xi,\cdot)$ and $\tilde{v}((f_t)_t)(x,\xi,\cdot)$ as in Definition \ref{definizione v tilde aa}. Then, for any $n\in\N,n\neq 0$ the \emph{torsion at finite time $n$} is
\begin{equation}\label{intro torsion tempi discreti}
Torsion_n((f_t)_t,x,\xi) := \dfrac{1}{n}\left( \tilde{v}((f_t)_t)(x,\xi,n)-\tilde{v}((f_t)_t))(x,\xi,0) \right).
\end{equation}
\end{definizione}

\begin{definizione}\label{definizione_torsion}
Let $S$ and $f$ be as above. Let $x\in S$. Assume that the quantity $Torsion_n((f_t)_t,x,\xi)$ converges as $n\rightarrow +\infty$ for some $\xi\in T_xS\setminus\{0\}$. The \emph{torsion of the orbit of $x$} is then
\begin{equation}\label{torsione def limite}
Torsion((f_t)_t,x) := \lim_{n\rightarrow+\infty}Torsion_n((f_t)_t,x,\xi).
\end{equation}
\end{definizione}

\noindent Whenever the limit exists, the previous quantity does not depend on the chosen lift of the angle function (see Proposition \ref{proprieta torsione}) or on the non zero vector of the tangent space (see Proposition \ref{torsione cambia vettore stima}).\\
\begin{definizione}\label{torsionemisura}
Let $S$ and $f$ be as above. Let $\mu$ be an $f$-invariant Borel probability measure on $S$. Assume that $\mu$ or $(f_t)_t$ has compact support\footnote{\noindent By asking that $(f_t)_t$ has compact support, we demand that for any $t\in[0,1]$ the support of $f_t$ is in a compact set, independent from $t$.}. Then, the \emph{torsion of the measure $\mu$} is
\begin{equation}
Torsion((f_t)_t,\mu) := \int_S Torsion((f_t)_t,x)d\mu(x).
\end{equation}
\end{definizione}
\begin{remark}
This integral is well defined. Indeed
$$
Torsion((f_t)_t,x) = \lim_{n\rightarrow+\infty} Torsion_n((f_t)_t,x,\xi)
$$
and
$$
Torsion_n((f_t)_t,x,\xi) = \dfrac{1}{n}\sum_{i=0}^{n-1} Torsion_1\left( (f_t)_t,f^i(x),Df^i(x)\xi\right)=
$$
$$
=\dfrac{1}{n}\sum_{i=0}^{n-1} Torsion_1((f_t)_t,\cdot,\cdot)\circ f_*^i(x,\xi)
$$
where we set
\begin{equation}
\begin{split}
f_*:&T^1S\rightarrow T^1S \\
(x,\xi)&\mapsto \left(f(x),\dfrac{Df(x)\xi}{\norma{Df(x)\xi}}\right).
\end{split}
\end{equation}
Lift $\mu$ to $\mu_*$, a $f_*$-invariant Borel probability measure on $T^1S$, and notice that 
\[
Torsion_1((f_t)_t,\cdot,\cdot)\in L^1(\mu_*)
\]
thanks to the assumption on the support of $\mu$ or of $(f_t)_t$. We deduce by Birkhoff's Ergodic Theorem that the function $Torsion((f_t)_t,\cdot)$ is defined $\mu$-a.e. and in $L^1(\mu)$.
\end{remark}


\indent In the setting of $\R^2$ we refer to \cite{beguin} to introduce the notion of \emph{linking number}.
\begin{notazione}\label{notation link r2}
The counterclockwise orientation of $\R^2$ is considered. Moreover, we fix the constant vector field $X=(1,0)$.
\end{notazione}
\begin{definizione}\label{definizionelinking}
Let $(F_t)_t$ be an isotopy in $\text{Diff}^1(\R^2)$ joining the identity to $F_1=F$.\\
\noindent Let us denote $\Delta:=\{(z_1,z_2)\in\R^4 :\ z_1=z_2\}$ and define the  function
\begin{equation}
\begin{split}
u((F_t)_t):&(\R^4\setminus\Delta)\times\R\rightarrow\mathbb{T}\\
&(z_1,z_2,t)\mapsto \theta\left( (1,0), F_t(z_2)-F_t(z_1) \right).
\end{split}
\end{equation}
\noindent Fix $(z_1,z_2)\in\R^4\setminus\Delta$ and consider $\tilde{u}((F_t)_t)(z_1,z_2,\cdot):\R\rightarrow\R$, a continuous determination of the angle function $u((F_t)_t)(z_1,z_2,\cdot)$.\\
\noindent For any $n\in\N, n\neq 0$, the \emph{linking number of $z_1$ and $z_2$ at finite time $n$} is
\begin{equation}
Linking_n((F_t)_t,z_1,z_2) := \dfrac{1}{n}\left( \tilde{u}((F_t)_t)(z_1,z_2,n)-\tilde{u}((F_t)_t)(z_1,z_2,0) \right).
\end{equation}
\noindent The \emph{linking number of the orbits of $z_1$ and $z_2$} is
\begin{equation}\label{linking number def limite}
Linking((F_t)_t,z_1,z_2):=\lim_{n\rightarrow+\infty}Linking_n((F_t)_t,z_1,z_2)
\end{equation}
whenever the limit exists.
\end{definizione}

\begin{remark}\label{remark esistenza linking misura}
Let $F$ be as in Definition \ref{definizionelinking}. Let $x\in\R^2$ and let $\mu$ be a $F$-invariant Borel probability measure on $\R^2$ with compact support. Then, for $\mu$-almost every $x\in\R^2$, the linking number $Linking((F_t)_t,x,y)$ exists for $\mu$-almost every $y\in\R^2\setminus\{x\}$. Indeed
$$
Linking((F_t)_t,x,y) = \lim_{n\rightarrow +\infty}\dfrac{1}{n}\sum_{i=0}^{n-1} Linking_1((F_t)_t, F^i(x), F^i(y))=
$$
$$
= \lim_{n\rightarrow +\infty}\dfrac{1}{n} \sum_{i=0}^{n-1} Linking_1((F_t)_t,\cdot,\cdot)\circ F_*^i(x,y)
$$
where
$$
F_*:\R^4\setminus\Delta\rightarrow \R^4\setminus\Delta
$$
$$
(x,y)\mapsto (F(x),F(y)).
$$
\noindent Considering the product measure $\mu\times\mu$ on $\R^4\setminus\Delta$, which is $F_*$-invariant, observe that
$$
Linking_1((F_t)_t,\cdot,\cdot)\in L^1(\mu\times\mu)
$$
since $\mu$ has compact support. Then, Birkhoff's Ergodic Theorem tells us that the function $Linking((F_t)_t,\cdot,\cdot)$ is defined $\mu\times\mu$-almost everywhere and it is in $L^1(\mu\times\mu)$. By Fubini's theorem, for $\mu$-almost every $x\in\R^2$ the function $Linking((F_t)_t,x,\cdot)$ is defined $\mu$-almost everywhere.
\end{remark}


\subsection{First properties of Torsion and Linking number and independence of Torsion from the choice of vector}\label{subsection proprieta 1}
The following propositions highlight some interesting properties of torsion and linking number concerning the choice of the continuous determination, of the tangent vector (for the torsion) and of the isotopy.\\
\noindent Let $S$ be a parallelizable surface and fix an orientation and a reference continuous vector field $X:S\rightarrow TS$ of unitary norm on it. For any $x\in S$, $(X(x),JX(x))$ provides a direct basis of the tangent space.\\
\noindent Let $(f_t)_t$ be an isotopy in $\text{Diff}^1(S)$ joining the identity to $f_1=f$.
\begin{proposizione}\label{proprieta torsione}
For any $(x,\xi)\in TS_*$ the quantities
$$
Torsion_n((f_t)_t,x,\xi)\qquad\forall n\in\N,n\neq 0
$$
$$
Torsion((f_t)_t,x) = \lim_{n\rightarrow+\infty}Torsion_n((f_t)_t,x,\xi)\qquad\text{when it exists}
$$
do not depend on the choice of the continuous determination of the angle function $v((f_t)_t)(x,\xi,\cdot)$.\\
\noindent Let $(g_t)_t$ be another isotopy joining the identity to $f$ and assume that $S$ is connected. There exists an integer $k\in\Z$ independent from $x\in S$ and $\xi\in T_xS\setminus\{0\}$ so that
$$
Torsion_n((f_t)_t,x,\xi)=Torsion_n((g_t)_t,x,\xi)+2\pi k\qquad\forall n\in\N,n\neq 0
$$
$$
Torsion((f_t)_t,x)=Torsion((g_t)_t,x)+2\pi k.
$$
\end{proposizione}
\noindent The proof is an immediate consequence of the continuity of the involved functions and the property of $f$ of being isotopic to the identity.
\begin{proposizione}\label{PROP IMPO vV}
Fix $x\in S$ and define the following functions
$$
\Pi:\R\rightarrow T_xS
$$
$$
s\mapsto \cos(s) X(x) + \sin(s) JX(x)
$$
and
\begin{equation}
\begin{split}
w((f_t)_t,x)&:\R\times\R\rightarrow\mathbb{T}\\
(s,t)&\mapsto\theta\left( X(f_t(x)), Df_t(x)\Pi(s) \right).
\end{split}
\end{equation}
\noindent Then, there exists a unique continuous determination of $w((f_t)_t,x)$, denoted as $W:\R\times\R\rightarrow\R$, such that $W(0,0)=0$. Moreover
\begin{itemize}
\item[$(i)$] $W(\cdot,0)=Id_{\R}(\cdot)$.
\item[$(ii)$] For any $t\in\R$, $W(\cdot,t)$ is an increasing homeomorphism of $\R$.
\item[$(iii)$] For any $s,t\in\R$, $W(s+\pi,t)=W(s,t)+\pi$.
\end{itemize}
\end{proposizione}
\begin{proof}
By the continuity of the isotopy with respect to the compact-open ${\cal{C}}^1$ topology, the function $w((f_t)_t,x)$ is continuous. There is a unique continuous determination $W$ such that $W(0,0)=0$, since by fixing the value of $W$ in a point we are selecting the lift.
\[
\xymatrix@C-1pc{
	\R\times\R \ar[rr]^{W} \ar[dr]_{w} && \R \ar[dl]^{p} \\
	&\mathbb{T}
}
\]

\noindent Notice that
\begin{equation*}
\begin{split}
W(\cdot,0)&:\R\rightarrow\R\\
s&\mapsto W(s,0)
\end{split}
\end{equation*}
is a lift of $\R\ni s\mapsto w((f_t)_t,x)(s,0)=\theta(X(x),\Pi(s))=p(s)$. Since $W(0,0)=0$, $W(\cdot,0)$ is the identity of $\R$.\\
\noindent Let us introduce the following function
$$
\bar{\Pi}:\mathbb{T}\rightarrow T_x S
$$
$$
\xi\mapsto \cos(\xi)X(x) +\sin(\xi)JX(x).
$$
\noindent For any fixed $t\in\R$, the function $W(\cdot,t) :\R\rightarrow\R$ is a continuous lift of the angle function
\begin{equation*}
\begin{split}
m(\cdot,t)&:\mathbb{T}\rightarrow\mathbb{T}\\
\xi &\mapsto \theta\left( X(f_t(x)), Df_t(x)\bar{\Pi}(\xi) \right).
\end{split}
\end{equation*}
\noindent As $Df_t(x)$ is linear and preserves the orientation, $m(\cdot,t)$ is an orientation preserving circle homeomorphism such that $m(\xi+\pi,t)=m(\xi,t)+\pi$. Hence, its lift $W(\cdot,t)$ is an increasing homeomorphism of $\R$.\\
\noindent The functions $(s,t)\mapsto W(s,t)+\pi$ and $(s,t)\mapsto W(s+\pi,t)$ are two lifts of $(s,t)\mapsto m(s,t)+\pi$ that coincide for $(s,t)=(0,0)$, hence $W(s+\pi,t)=W(s,t)+\pi$, i.e. $W(\cdot,t)$ commutes with the translation of $\pi$ for any $t\in\R$.\\
\end{proof}


\begin{proposizione}\label{torsione cambia vettore stima}
Let $x\in S$. Assume that for some $\xi\in T_xS\setminus\{0\}$ the quantity $Torsion_n((f_t)_t,x,\xi)$ converges as $n\rightarrow+\infty$. Then, the torsion of the orbit of $x$ does not depend on the choice of the tangent vector. In other words, for any vector $\delta\in T_xS\setminus\{0\}$ it holds
$$
Torsion((f_t)_t,x)=\lim_{n\rightarrow+\infty}Torsion_n((f_t)_t,x,\xi)=\lim_{n\rightarrow+\infty}Torsion_n((f_t)_t,x,\delta).
$$
\end{proposizione}
\begin{proof}
Consider $\xi,\delta\in T_xS\setminus\{0\}$ and assume that $\lim_{n\rightarrow+\infty}Torsion_n((f_t)_t,x,\xi)$ exists. Then, also $\lim_{n\rightarrow+\infty}Torsion_n((f_t)_t,x,\delta)$ exists and it coincides with the previous one.\\
The result easily follows once we prove that
$$
\lim_{n\rightarrow+\infty}\abs{Torsion_n((f_t)_t,x,\xi)-Torsion_n((f_t)_t,x,\delta)}=0.
$$
\begin{lemma}\label{lemma vettore n acceso}
Fix $x\in S$. For $n\in\N,n\neq 0$ and for $\xi,\delta\in T_xS\setminus\{0\}$ it holds
\begin{equation}
\abs{Torsion_n((f_t)_t,x,\xi)-Torsion_n((f_t)_t,x,\delta)}<\dfrac{\pi}{n}.
\end{equation}
\end{lemma}
\begin{proof}
The quantity
$$
\abs{Torsion_n((f_t)_t,x,\xi)-Torsion_n((f_t)_t,x,\delta)}
$$
can be written as
$$
\dfrac{1}{n}\abs{\left( \tilde{v}((f_t)_t)(x,\xi,n)-\tilde{v}((f_t)_t)(x,\delta,n) \right)-\left( \tilde{v}((f_t)_t)(x,\xi,0)-\tilde{v}((f_t)_t)(x,\delta,0) \right)}.
$$
These quantities do not depend on the chosen determination of the angle function $v$. 
Concerning the relative position of the vectors $\xi,\delta$, four cases can occur:
\begin{equation*}
\tilde{v}((f_t)_t)(x,\xi,0) - \tilde{v}((f_t)_t)(x,\delta,0)  \begin{cases}
=2\pi k \qquad\text{if $\xi,\delta$ are positively colinear}\\
=\pi+2\pi k \qquad\text{if $\xi,\delta$ are negatively colinear}\\
\in(0,\pi)+2\pi k \qquad\text{if $(\xi,\delta)$ is a direct basis}\\
\in(\pi,2\pi) + 2\pi k \qquad\text{if $(\xi,\delta)$ is an indirect basis}.
\end{cases}
\end{equation*}
At any time, the same four cases can occur and
\begin{equation*}
\tilde{v}((f_t)_t)(x,\xi,t) - \tilde{v}(f_t)_t)(x,\delta,t)\begin{cases}
=2\pi k \qquad\text{if $\xi,\delta$ are positively colinear}\\
=\pi+2\pi k \qquad\text{if $\xi,\delta$ are negatively colinear}\\
\in(0,\pi)+2\pi k \qquad\text{if $(\xi,\delta)$ is a direct basis}\\
\in(\pi,2\pi) + 2\pi k \qquad\text{if $(\xi,\delta)$ is an indirect basis}.
\end{cases}
\end{equation*}
where the integer $k\in\Z$ is the same for any $t$.\\
\noindent This holds in particular for $t=n$ and, checking all the possible cases, we obtain
$$
\dfrac{1}{n}\abs{\left( \tilde{v}((f_t)_t)(x,\xi,n)-\tilde{v}((f_t)_t)(x,\delta,n) \right) - \left( \tilde{v}((f_t)_t)(x,\xi,0)-\tilde{v}((f_t)_t)(x,\delta,0) \right)}<\dfrac{\pi}{n}.
$$
\end{proof}
\noindent From Lemma \ref{lemma vettore n acceso} we conclude since
$$
0\leq\lim_{n\rightarrow+\infty}\abs{Torsion_n((f_t)_t,x,\xi)-Torsion_n((f_t)_t,x,\delta)}\leq\lim_{n\rightarrow+\infty}\dfrac{\pi}{n}=0.
$$
\end{proof}

\indent Analogous properties hold true for the linking number, whose proof is similar to that of Proposition \ref{torsione cambia vettore stima}.\\
\begin{proposizione}\label{prprieta linking number}
Let $(F_t)_{t\in \R}$ be an isotopy in $\text{Diff}^1(\R^2)$ joining the identity $F_0=Id_{\R^2}$ to $F_1=F$. For any points $z_1,z_2\in\R^2,z_1\neq z_2$ the quantities
$$
Linking_n((F_t)_t,z_1,z_2)\qquad\forall n\in\N,n\neq 0
$$
$$
Linking((F_t)_t,z_1,z_2)\qquad\text{when it exists}
$$
do not depend on the choice of the continuous determination of the angle function $u((F_t)_t)(z_1,z_2,\cdot)$.\\
\noindent Let $(G_t)_t$ be another isotopy joining the identity to $F$. Then, there exists an integer $k\in\Z$ independent from the points $z_1,z_2\in\R^2$ so that
$$
Linking_n((F_t)_t,z_1,z_2)=Linking_n((G_t)_t,z_1,z_2)+2\pi k\qquad\forall n\in\N,n\neq 0
$$
$$
Linking((F_t)_t,z_1,z_2)=Linking((G_t)_t,z_1,z_2)+2\pi k.
$$
\end{proposizione}

\subsection{Independence of Torsion from the isotopy for diffeomorphisms on $\mathbb{A}$ and $\mathbb{T}^2$}\label{subsection proprieta 2}

In \cite{beguin}, B\'eguin and Boubaker show that the torsion is independent from the choice of the isotopy both for an isotopy with compact support and for a diffeomorphism on the $2$-dimensional torus $\mathbb{T}^2$. In this Section, we prove the independence of the torsion from the isotopy for a ${\cal{C}}^1$ diffeomorphism over the annulus (with no further hypothesis on its support).
\begin{notazione}
Until the end of the paper (if not specified), we will consider as parallelizable surface the annulus $\mathbb{A}=\mathbb{T}\times\R$. Let $(f_t)_t$ be an isotopy in $\text{Diff}^1(\mathbb{A})$ joining $Id_{\mathbb{A}}$ to $f_1=f$. Let us fix the counterclockwise orientation on $\R^2$ and consider as continuous never-vanishing vector field $X$ the constant one $(1,0)$.\\
\noindent Let $(F_t)_t$ be the isotopy obtained as the lift of $(f_t)_t$ such that $F_0=Id_{\R^2}$. It joins the identity $Id_{\R^2}$ to $F$, where $F:\R^2\rightarrow\R^2$ is a lift of $f$. We then remark that for any time $t$ and for any $z=(x,y)\in\R^2$ it holds
\begin{equation}\label{condizione periodicita lift}
F_t(x+2\pi,y) = F_t(x,y) + (2\pi,0).
\end{equation}
\end{notazione}
\noindent As an intermediate step, we first show that the linking number in the lifted setting does not depend on the choice of the isotopy.
\begin{proposizione}\label{linking non depende chosen isotopy}
Let $(f_t)_t,(g_t)_t$ be two different isotopies in $\text{Diff}^1(\mathbb{A})$ joining $Id_{\mathbb{A}}$ to $f_1=g_1=f$. Let $(F_t)_t,(G_t)_t$ in $\text{Diff}^1(\R^2)$ be lifts of the isotopies $(f_t)_t,(g_t)_t$ such that $F_0=G_0=Id_{\R^2}$.\\
\noindent Then for any $z_1,z_2\in\R^2$, $z_1\neq z_2$ it holds
$$
Linking_1((F_t)_t,z_1,z_2) = Linking_1((G_t)_t,z_1,z_2)
$$
and hence, whenever the limit exists, $Linking((F_t)_t,z_1,z_2)=Linking((G_t)_t,z_1,z_2)$.
\end{proposizione}
\begin{proof}
Recalling the definition of the diagonal in $\R^4$, that is
$$
\Delta := \{ \left( (x,y),(x',y') \right)\in\R^4 :\ (x,y) = (x',y') \},
$$
we define the following functions
\begin{equation*}
\begin{split}
Linking_1((F_t)_t):&(\R^2\times\R^2)\setminus\Delta\rightarrow\R \\
&(z,z')\mapsto Linking_1((F_t)_t,z,z')
\end{split}
\end{equation*}
and
\begin{equation*}
\begin{split}
Linking_1((G_t)_t):&(\R^2\times\R^2)\setminus\Delta\rightarrow\R\\
&(z,z')\mapsto Linking_1((G_t)_t,z,z').
\end{split}
\end{equation*}

\noindent Both these functions are continuous ones. Moreover, for any $(z,z')\in(\R^2\times\R^2)\setminus\Delta$ there exists $k=k_{z,z'}\in\Z$ such that
$$
Linking_1((F_t)_t,z,z')=Linking_1((G_t)_t,z,z')+2\pi k.
$$
\noindent Since $(\R^2\times\R^2)\setminus\Delta$ is connected, the integer $k\in\Z$ does not depend on the points $(z,z')$ of $(\R^2\times\R^2)\setminus\Delta$.\\

\indent Consider then points $z\neq z'$ such that $z'=z+(2\pi,0)$. To fix the ideas, let us choose $z=(0,0),z'=(2\pi,0)$. Because of \eqref{condizione periodicita lift}, it holds that
\begin{equation}
Linking_1((F_t)_t,z,z')=Linking_1((G_t)_t,z,z')=0.
\end{equation}
\noindent By this observation, we conclude that $k=0$, i.e. the linking number does not depend on the chosen isotopy. 
\end{proof}

The next proposition proves that the definition of torsion for a ${\cal{C}}^1$ diffeomorphism $f:\mathbb{A}\rightarrow\mathbb{A}$ isotopic to the identity is independent from the choice of the isotopy.

\begin{proposizione}\label{indipendenza torsione isotopia}
Let $(f_t)_t,(g_t)_t$ be two different isotopies in $\text{Diff}^1(\mathbb{A})$ joining the identity $Id_{\mathbb{A}}$ to $f_1=g_1=f$.\\
\noindent Then for any $\bar{z}\in\mathbb{A}$ and for any $\xi\in T_{\bar{z}}\mathbb{A}\setminus\{0\}$
\begin{equation}
Torsion_1((f_t)_t,\bar{z},\xi) = Torsion_1((g_t)_t,\bar{z},\xi).
\end{equation}
Moreover
\begin{equation}
Torsion((f_t)_t,\bar{z}) = Torsion((g_t)_t,\bar{z})
\end{equation}
whenever the limit exists.
\end{proposizione}
\begin{proof}
Let $(F_t)_t$ and $(G_t)_t$ be the corresponding lifts of the isotopies $(f_t)_t$ and $(g_t)_t$ to the plane $\R^2$ such that $F_0=G_0=Id_{\R^2}$. Let $z\in\R^2$ and $\xi\in T_{z}\R^2\setminus\{0\}\cong T_{\bar{z}}\mathbb{A}\setminus\{0\}$. Thanks to the choice of the trivialization, denoting as $\bar{z}\in\mathbb{A}$ the projection of $z$ on the annulus, it holds
$$
Torsion_1((F_t)_t,z,\xi)=Torsion_1((f_t)_t,\bar{z},\xi)
$$
and
$$
Torsion_1((G_t)_t,z,\xi) = Torsion_1((g_t)_t,\bar{z},\xi).
$$
\noindent By Proposition \ref{proprieta torsione} it holds
\begin{equation}\label{section 4 costante k}
Torsion_1((F_t)_t,z,\xi) = Torsion_1((G_t)_t,z,\xi) + 2\pi k
\end{equation}
where $k\in\Z$ does not depend on the point or on the vector since $\R^2$ is connected.\\
\noindent Recall the functions $v,u$, used in Definitions \ref{definizione_torsion} and \ref{definizionelinking}:
\begin{equation*}
\begin{split}
v((F_t)_t)(z,\xi,\cdot):&[0,1]\rightarrow\mathbb{T}\\
&t\mapsto\theta\left( (1,0), DF_t(z)\xi \right)
\end{split}
\end{equation*}
and
\begin{equation*}
\begin{split}
u((F_t)_t)(z,z',\cdot):&[0,1]\rightarrow\mathbb{T}\\
&t\mapsto\theta\left( (1,0), F_t(z')-F_t(z) \right).
\end{split}
\end{equation*}
\noindent Let us look at $z'=z+\xi$. 
\noindent Parametrize the segment $[z,z+\xi]$ by setting for any $s\in[0,1]$
$$
z(s):=z+s\xi.
$$
\noindent Modify now the definitions of functions $u,v$ in the following way:
\begin{equation}\label{linking_notation_modify1}
\begin{split}
u((F_t)_t):&[0,1]\times[0,1]\rightarrow\mathbb{T}\\
&(s,t)\mapsto \theta\left((1,0), F_t(z(s))-F_t(z)\right)\qquad s\neq 0 \\
&(0,t)\mapsto\theta\left((1,0), DF_t(z)\xi\right)
\end{split}
\end{equation}
and
\begin{equation}\label{torsion_notation_modify}
\begin{split}
v((F_t)_t):&[0,1]\times[0,1]\rightarrow\mathbb{T}\\
&(s,t)\mapsto\theta\left((1,0), DF_t(z(s))\xi\right).
\end{split}
\end{equation}
\noindent Observe that both $v(s,t)$ and $u(s,t)$ are continuous functions, by the continuity of the isotopy with respect to the weak ${\cal{C}}^1$ topology in $\text{Diff}^1(\R^2)$.\\
\noindent Since the definition of $u((F_t)_t)$ coincides with that of $v((F_t)_t)$ for $s=0$ and since $u((F_t)_t)$ is continuous, for any time $t$ we have that $v((F_t)_t)(0,t)=u((F_t)_t)(0,t)=\lim_{s\rightarrow 0^+}u((F_t)_t)(s,t)$.\\
The definitions of torsion and linking number do not depend on the chosen lift. So we select continuous determinations $\tilde{v}((F_t)_t)$ and $\tilde{u}((F_t)_t)$ such that $\tilde{v}((F_t)_t)(0,t)=\tilde{u}((F_t)_t)(0,t)$ for any time $t$ and similarly $\tilde{v}((G_t)_t)(0,t)=\tilde{u}((G_t)_t)(0,t)$.\\

\noindent By Proposition \ref{linking non depende chosen isotopy} for any $s\in[0,1]$ it holds
$$
\tilde{u}((F_t)_t)(s,1)-\tilde{u}((F_t)_t)(s,0)=\tilde{u}((G_t)_t)(s,1)-\tilde{u}((G_t)_t)(s,0).
$$
Passing to the limit for $s$ going to $0^+$, we obtain
$$
\tilde{v}((F_t)_t)(0,1)-\tilde{v}((F_t)_t)(0,0)=\tilde{v}((G_t)_t)(0,1)-\tilde{v}((G_t)_t)(0,0),
$$
that is
$$
Torsion_1((F_t)_t,z,\xi) = Torsion_1((G_t)_t,z,\xi).
$$
We conclude that the integer $k$ in \eqref{section 4 costante k} is null.
\end{proof}

\indent With the same techniques, it can be shown that also for a ${\cal{C}}^1$ diffeomorphism over the torus $\mathbb{T}^2$ isotopic to the identity the torsion is independent from the choice of the isotopy. Actually, this independence has been already remarked by B\'eguin and Boubaker in Section 2 in \cite{beguin}.\\

\section{Relation between Torsion and Linking number}\label{sezione relazione link tors}
In \cite{beguin}, the authors provide conditions for which the existence of two points with non-zero linking number implies the existence of a point with non-zero torsion. However, the value and even the sign of the linking number and of the torsion can be different.\\
\noindent Let $x,y\in\R^2$ be points with linking value $l$. We prove the existence of a point with torsion value exactly $l$. In addition we locate such a point on the segment joining $x$ and $y$. We remark that this result can be applied also to the zero value case, since it does not depend on the value of the linking number.

\begin{notazione}
Consider an isotopy $(F_t)_t:[0,1]\rightarrow \text{Diff}^1(\R^2)$ joining the identity to $F_1=F$. With the notation $(F_t)_t$ we refer also to the extended isotopy. We refer to the setting presented in Notation \ref{notation link r2}: we fix the counterclockwise orientation and we use the vector field $X=(1,0)$.\\
\noindent Given two points $x,y\in\R^2$, $x\neq y$, the notation $[x,y]$ refers to the segment joining the points.\\
\noindent Denote a point of the segment as $z(s) := sy + (1-s)x$ for $s\in[0,1]$.
\end{notazione}
\indent The main result concerning linking number and torsion at finite time $t=1$ is then Theorem \ref{thm intro 1}.

\noindent\textit{Sketch of the proof of Theorem \ref{thm intro 1}.}
By contradiction we assume that there is no point $z\in[x,y]$ such that $Torsion_1((F_t)_t,z,y-x)=l$. Then, by the continuity of the function $z\mapsto Torsion_1((F_t)_t,z,y-x)$ and by the connectedness of the segment, one of the following cases occur:
\begin{itemize}
\item[$(i)$] for any $z\in[x,y]$ it holds $Torsion_1((F_t)_t,z,y-x)<l$;
\item[$(ii)$] for any $z\in[x,y]$ it holds $Torsion_1((F_t)_t,z,y-x)>l$.
\end{itemize}
In Section \ref{sezione prova linking-torsion} we show that case $(i)$ leads to a contradiction. Similarly, case $(ii)$ cannot even occur.
\hfill\qed
\\

\indent A modification of the involved isotopy and the use of Theorem \ref{thm intro 1} easily adapt this result for any finite time $n\in\N$. We keep the same notation of Theorem \ref{thm intro 1}.

\begin{corollario}\label{corollario_nacceso_link}
Assume that there exist $n\in\N, n\neq 0$ and $x,y\in\R^2$, $x\neq y$, such that $Linking_n((F_t)_t,x,y)=l\in\R$.\\
\noindent Then there exists a point $z\in[x,y]$ such that
\begin{equation}
Torsion_n\left((F_t)_t,z,y-x\right)=l.
\end{equation}
\end{corollario}
\begin{proof}
We are interested in the time interval $[0,n]$. Define the isotopy $(G_t)_{t\in[0,1]}:=(F_{nt})_{t\in[0,1]}$.\\
\noindent Hence we are time-reparametrizing the initial isotopy. It holds $u((G_t)_t,x,y)(t)=u((F_t)_t,x,y)(nt)$. Then, $\tilde{u}((G_t)_t,x,y)(t)$ and $\tilde{u}((F_t)_t,x,y)(nt)$ denote continuous determinations of the same angle function. Since the (finite time) linking number is independent from the choice of the lift (see Proposition \ref{prprieta linking number}), we refer to $\tilde{u}((G_t)_t,x,y)(t)$.\\
\noindent The hypothesis $Linking_n((F_t)_t,x,y) = l$ is then equivalent to ask that $Linking_1((G_t)_t,x,y)=nl$.\\
\noindent By Theorem \eqref{thm intro 1}, there exists $z\in[x,y]$ such that $Torsion_1((G_t)_t,z,y-x)=nl$. For such a $z$ it also holds
\begin{equation}
Torsion_n\left((F_t)_t,z,y-x\right)=l
\end{equation}
and this concludes the proof.
\end{proof}

\noindent We wonder if any such relation is satisfied between asymptotic torsion and asymptotic linking number: can any results as above hold true even when considering \eqref{torsione def limite} in Definition \ref{definizione_torsion} and \eqref{linking number def limite} in Definition \eqref{definizionelinking}?\\
\noindent The answer is positive looking at torsion of $F$-invariant measures, instead of orbits.

\begin{corollario}\label{prop_relazione}
Assume that there exist two points $x,y\in\R^2,x\neq y$ such that $Linking((F_t)_t,x,y)=l\in\R$. Suppose that $\bigcup_{n\in\N}F^n([x,y])$ is relatively compact.\\
\noindent Then there exists a $F$-invariant probability measure $\mu$ such that
\begin{equation*}
Torsion((F_t)_t,\mu)=l.
\end{equation*}
\noindent Moreover, there exist points with torsion greater or equal $l$ and also points with torsion smaller or equal $l$.
\end{corollario}

\begin{remark}
If $F$ has compact support, then $\bigcup_{n\in\N}F^n([x,y])$ is always relatively compact.
\end{remark}

\begin{proof}
From our hypothesis
\begin{equation*}
l=Linking((F_t)_t,x,y)=\lim_{n\rightarrow +\infty}Linking_n((F_t)_t,x,y)
\end{equation*}
For any fixed $n\in\N$, denote $l_n:=Linking_n((F_t)_t,x,y)$. By Corollary \eqref{corollario_nacceso_link} there exists $z_n\in[x,y]$ such that
\begin{equation*}
Torsion_n\left((F_t)_t,z_n,y-x\right)=l_n.
\end{equation*}
The notation $\xi$ refers to the vector $y-x$. Consider the following probability measures on the unitary tangent bundle $T^1\R^2$:
\begin{equation}
\tilde{\mu}_n:=\dfrac{1}{n}\sum_{i=0}^{n-1}\delta_{\left(F^i(z_n),\frac{DF^i(z_n)\xi}{\norma{DF^i(z_n)\xi}}\right)}
\end{equation}
where $\delta_{(x,v)}$ denotes the Dirac measure centered on $(x,v)$ in $T^1\R^2$.
\noindent All the supports of these measures $\tilde{\mu}_n$ are contained in the same set
$$
T_{\mathcal{K}}^1\R^2
$$
where
$$
\mathcal{K}:= \overline{\bigcup_{i\in\N}F^i([x,y])}.
$$
From the hypothesis, $\mathcal{K}$ is compact and so is $T_{\mathcal{K}}^1\R^2$.\\
\noindent Up to subsequences, the sequence $(\tilde{\mu}_n)_n$ converges to a probability measure $\tilde{\mu}$ on $T^1\R^2$ which is invariant with respect to the dynamics on the unitary tangent bundle inherited from $F$. The projection $\mu$ of $\tilde{\mu}$ on $\R^2$ is $F$-invariant as well.\\
\noindent Finally, refering to Definition \eqref{torsionemisura} with respect to $\mu$, we have
\begin{equation*}
\begin{split}
Torsion((F_t)_t,\mu) &= \int_{\R^2}Torsion((F_t)_t,x)d\mu(x) = \\
\int_{T^1\R^2}Torsion((F_t)_t,x)d\tilde{\mu}(x,v) &\overset{*}{=}\int_{T^1\R^2}Torsion_1((F_t)_t,x,v)d\tilde{\mu}(x,v)=\\
=\lim_{n\rightarrow+\infty}\int_{T^1\R^2}&Torsion_1((F_t)_t,x,v)d\tilde{\mu}_n(x,v) = \\
=\lim_{n\rightarrow+\infty}\dfrac{1}{n}\sum_{i=0}^{n-1}&Torsion_1\left((F_t)_t,F^i(z_n),\dfrac{DF^i(z_n)\xi}{\norma{DF^i(z_n)\xi}}\right)= \\
=\lim_{n\rightarrow+\infty}Torsion&_n((F_t)_t,z_n,\xi)=\lim_{n\rightarrow+\infty}l_n=l.
\end{split}
\end{equation*}
Equality $*$ is a consequence of Birkhoff's Ergodic Theorem applied to the framework where
$$
F^*:(T^1\R^2,\tilde{\mu})\rightarrow(T^1\R^2,\tilde{\mu})
$$
$$
(x,\xi)\mapsto F^*(x,\xi)=\left(F(x),\dfrac{DF(x)\xi}{\norma{DF(x)\xi}}\right)
$$
is a measure-preserving transformation and $Torsion_1((F_t)_t,\cdot,\cdot)\in L^1(T^1\R^2,\tilde{\mu})$. The time average $Torsion((F_t)_t,\cdot)$ does not depend on the choice of the tangent vector (see Proposition \ref{torsione cambia vettore stima}) and, by Birkhoff's Ergodic Theorem (see Theorem 4.1.2 in \cite{katok}), it exists $\tilde{\mu}$-a.e., is measurable, $F^*$-invariant and such that
$$
\int_{T^1\R^2}Torsion((F_t)_t,x)d\tilde{\mu}(x,v) = \int_{T^1\R^2}Torsion_1((F_t)_t,x,v)d\tilde{\mu}(x,v).
$$

\noindent As an outcome, there exist points with torsion greater or equal $l$ and also points with torsion smaller or equal $l$.\\
\noindent Arguing by contradiction, suppose that every $x\in\R^2$ has $Torsion((F_t)_t,x)$ strictly greater than $l$. Then
\begin{equation*}
l=Torsion((F_t)_t,\mu)=\int_{\R^2}Torsion((F_t)_t,x)d\mu(x)>\int_{\R^2}l\ d\mu(x)=l.
\end{equation*}
This provides the required contradiction. Analogous argument holds assuming that every point has torsion strictly less than $l$.
\end{proof}

\subsection{Some consequences over the torus $\mathbb{T}^2$}
Any diffeomorphism of the torus has an invariant measure with zero torsion: this result was already known by Matsumoto and Nakayama for $\mathcal{C}^{\infty}$ diffeomorphisms. We present here a simpler proof which works also with $\mathcal{C}^1$ diffeomorphisms. Therefore, we weaken the hypothesis required in \cite{matsumoto}.

\begin{notazione}
\noindent Let 
$$
\mathscr{P}:\R^2\rightarrow\mathbb{T}^2
$$
$$
(x,y)\mapsto \mathscr{P}(x,y) = (x\ mod\,2\pi,y\ mod\,2\pi)
$$
be the universal covering of $\mathbb{T}^2$. Denote as $\mathcal{P}(\mathbb{T}^2)$ the set of Borel probability measure over the torus $\mathbb{T}^2$. Fix the counterclockwise orientation and consider as reference vector field $X$ the constant one $(1,0)$.
\end{notazione}
\noindent Let us start by observing that in the case of torus diffeomorphisms the hypothesis of Corollary \ref{prop_relazione} are too strong. Therefore, we state the following
\begin{corollario}\label{prop_relazione_torus}
Let $(f_t)_t$ be an isotopy in $\text{Diff }^1(\mathbb{T}^2)$ joining $Id_{\mathbb{T}^2}$ to $f_1=f$. Let $(F_t)_t$ in $\text{Diff }^1(\R^2)$ be the lift of the isotopy $(f_t)_t$ such that $F_0=Id_{\R^2}$. Assume that there exist two points $x,y\in\R^2,x\neq y$ such that $Linking((F_t)_t,x,y)=l\in\R$. Then there exists a $f$-invariant probability measure $\mu\in \mathcal{P}(\mathbb{T}^2)$ such that $Torsion((f_t)_t,\mu)=l$. Moreover, there exist points in $\mathbb{T}^2$ with torsion greater or equal $l$ and also points with torsion smaller or equal $l$.
\end{corollario}
\noindent The proof of Corollary \ref{prop_relazione_torus} retraces the ideas of the proof of Corollary \ref{prop_relazione}.
\begin{proof}
As in the proof of Corollary \ref{prop_relazione}, denote $l_n=Linking_n((F_t)_t,x,y)$ and by hypothesis it holds $\lim_{n\rightarrow+\infty}l_n=l=Linking((F_t)_t,x,y)$. By Corollary \ref{corollario_nacceso_link} for any $n\in\N, n\neq 0$ there exists $z_n\in[x,y]$ such that $Torsion_n((F_t)_t,z_n,y-x)=l_n$.\\
\noindent Thanks to the choice of the trivialization we have that
$$
Torsion_n((f_t)_t,\mathscr{P}(z_n),y-x)=Torsion_n((F_t)_t,z_n,y-x)=l_n.
$$
For simplicity denote $\mathscr{P}(z_n)$ as $\bar{z}_n$. Consider now the probability measures on the unitary tangent bundle $T^1\mathbb{T}^2$:
$$
\tilde{\mu}_n:=\dfrac{1}{n}\sum_{i=0}^{n-1}\delta_{\left( f^i(\bar{z}_n),\frac{Df^i(\bar{z}_n)(y-x)}{\norma{Df^i(\bar{z}_n)(y-x)}} \right)}.
$$
Being $T^1\mathbb{T}^2$ compact, up to subsequences, $(\tilde{\mu}_n)_n$ converges to $\tilde{\mu}$ which is a probability measure on $T^1\mathbb{T}^2$. The measure $\tilde{\mu}$ is invariant with respect to the dynamics on $T^1\mathbb{T}^2$ and its projection on $\mathbb{T}^2$ $\mu\in\mathcal{P}(\mathbb{T}^2)$ is $f$-invariant.\\
\noindent Repeating the ideas in the proof of Corollary \ref{prop_relazione}, we have
$$
Torsion((f_t)_t,\mu) = \int_{\mathbb{T}^2} Torsion((f_t)_t,x)\ d\mu(x)=\int_{T^1\mathbb{T}^2}Torsion_1((f_t)_t,x,v)\ d\tilde{\mu}(x,v)=
$$
$$
=\lim_{n\rightarrow+\infty}\dfrac{1}{n}\sum_{i=0}^{n-1}Torsion_1\left( (f_t)_t,f^i(\bar{z}_n),\dfrac{Df^i(\bar{z}_n)(y-x)}{\norma{Df^i(\bar{z}_n)(y-x)}} \right)= \lim_{n\rightarrow+\infty}Torsion_n((f_t)_t,\bar{z}_n,y-x)=l.
$$
\noindent We easily deduce the existence of points in $\mathbb{T}^2$ with torsion greater or equal $l$ (respectively smaller or equal $l$).
\end{proof}

\noindent We then deduce as a corollary the result by Matsumoto and Nakayama discussed above.

\begin{corollario}
Let $(f_t)_t$ be an isotopy in $\text{Diff }^1(\mathbb{T}^2)$ joining the identity $Id_{\mathbb{T}^2}$ to $f_1=f$.\\
\noindent Then, there exists a $f$-invariant Borel probability measure $\mu\in\mathcal{P}(\mathbb{T}^2)$ of null torsion.
\end{corollario}
\begin{proof}
Let $(F_t)_t$ be the isotopy obtained as the lift of the isotopy $(f_t)_t$ such that $F_0=Id_{\R^2}$. For any point $(x,y)\in\R^2$
\begin{equation}\label{isotopia lift toro r2}
F_t(x+2\pi k_1,y+2\pi k_2)=F_t(x,y)+(2\pi k_1,2\pi k_2)\qquad\forall (k_1,k_2)\in\Z^2,\forall t\in\R_+.
\end{equation}
Consider now the points $z_1=(0,0),z_2=(2\pi,0)\in\R^2$. For a fixed $n\in\N,n\neq 0$ look at
\begin{equation*}
Linking_n((F_t)_t,z_1,z_2)=\dfrac{1}{n}\left( \tilde{u}((F_t)_t)(z_1,z_2,n) - \tilde{u}((F_t)_t)(z_1,z_2,0) ) \right).
\end{equation*}
\noindent Since \eqref{isotopia lift toro r2} holds for every $t\geq 0$, the vector $F_t((2\pi,0))-F_t((0,0))$ (in whose direction we are interested) remains horizontal and so $Linking_n((F_t)_t,z_1,z_2)\equiv\frac{0}{n}$. By the arbitrariness of $n\in\N$ we deduce that $Linking((F_t)_t,z_1,z_2)=0$. Applying Corollary \ref{prop_relazione_torus} to the points $z_1,z_2$, we conclude that there exists $\mu\in\mathcal{P}(\mathbb{T}^2)$ which is $f$-invariant and such that $Torsion((f_t)_t,\mu)=0$.
\end{proof}

\subsection{Proof of case $(i)$ of Theorem \ref{thm intro 1}}\label{sezione prova linking-torsion}

In this section we assume that case $(i)$ of the sketch of the proof of Theorem \ref{thm intro 1} (presented in Section \ref{sezione relazione link tors}) holds, that is for any $z\in[x,y]$ we have $Torsion_1((F_t)_t,z,y-x)<l=Linking_1((F_t)_t,x,y)$. We are going to find a contradiction, deducing that this case cannot occur.\\
\noindent By continuity of the function and by compactness of the segment, we assume that there exists $\varepsilon>0$ such that for any point in $[x,y]$
\begin{equation*}
Torsion_1\left((F_t)_t,z,y-x\right)<l-\varepsilon.
\end{equation*}


\begin{notazione}
Denote
\begin{equation*}
\xi := y-x
\end{equation*}
and parametrize the segment $[x,y]$ as follows:
\begin{equation*}
[0,1]\ni s\mapsto z(s):= sy+(1-s)x\in[x,y]\subset\R^2.
\end{equation*}
\end{notazione}

\begin{notazione}\label{notazione_di_modifica_definizione_u_v}





We use the notation introduced in \eqref{linking_notation_modify1} and \eqref{torsion_notation_modify} in order to modify the angle functions $u,v$. From these, we define linking number and torsion just along points of the segment $[x,y]$.\\
\noindent Since $u,v$ are continuous and for any $t$ it holds $u(0,t)=v(0,t)$, there exist continuous lifts $\tilde{u},\tilde{v}$ of the functions $u,v$, respectively, such that $\tilde{u}(0,t)=\tilde{v}(0,t)$.

\end{notazione}

By hypothesis for any $s\in[0,1]$
\begin{equation}\label{eq numeronumero}
Torsion_1((F_t)_t,z(s),\xi)<l-\varepsilon=Linking_1((F_t)_t,x,y)-\varepsilon.
\end{equation}
Refering to definitions \eqref{linking_notation_modify1} and \eqref{torsion_notation_modify}, inequality \eqref{eq numeronumero} becomes
\begin{equation}\label{estimate_du_torsionneg_et_linkpos}
\tilde{v}(s,1)-\tilde{v}(s,0)<\tilde{u}(1,1)-\tilde{u}(1,0)-\varepsilon
\end{equation}
for any $s\in[0,1]$.

\subsubsection{Modification of the isotopy $(F_t)_t$}

First, we modify the given isotopy $(F_t)_t$ to obtain an isotopy $(H_t)_t$ such that:
\begin{itemize}
\item the point $x$ is fixed for $(H_t)_t$, that is $H_t(x)=x$ for any $t$;
\item the linking number of $x,y$ with respect to $(H_t)_t$ is positive, while the torsion of any point of $[x,y]$ with respect to $(H_t)_t$ is negative.
\end{itemize}
In other words, we want to pass in a rotated and translated frame.

\begin{lemma}
Let $(F_t)_{t\in[0,1]}$ be an isotopy in $\text{Diff}^1(\R^2)$ joining $Id_{\R^2}$ to $F_1=F$. Consider $x,y\in\R^2$, $x\neq y$ such that, for a fixed $\varepsilon>0$, for any $s\in[0,1]$
\begin{equation}\label{equazione37inequality}
Torsion_1((F_t)_t,z(s),\xi)<Linking_1((F_t)_t,x,y)-\varepsilon.
\end{equation}

\noindent Then, there exists an isotopy $(H_t)_{t\in[0,1]}$ in $\text{Diff}^1(\R^2)$, such that:
\begin{itemize}
\item $H_0=Id_{\R^2}$ and $H:=H_1$;
\item for any $s\in[0,1]$
\begin{equation*}
Torsion_1((H_t)_t,z(s),\xi)\leq-\dfrac{\varepsilon}{2}<0<\dfrac{\varepsilon}{2}\leq Linking_1((H_t)_t,x,y);
\end{equation*}
\item $H_t(x)=x$ for any $t\in[0,1]$.
\end{itemize}
\end{lemma}

\begin{proof}
Define the following continuous function
\begin{equation}\label{def_Theta}
\begin{split}
\Theta : [0,1]&\rightarrow\R \\
\Theta(t) :=\sup_{s\in[0,1]} \left( \tilde{v}(s,t)-\tilde{v}(s,0)\right)&  = \max_{s\in[0,1]}\left( \tilde{v}(s,t)-\tilde{v}(s,0)\right)
\end{split}
\end{equation}

\noindent We remark that $\Theta(0)=0$. The new isotopy is then obtained as follows:
\begin{equation*}
(H_t)_{t\in[0,1]} := {\cal{R}}\left(x,-\Theta(t)-t\frac{\varepsilon}{2}\right)\circ\tau_{x-F_t(x)}\circ (F_t)_t
\end{equation*}
where ${\cal{R}}\left(x,\psi\right)$ denotes the rotation of angle $\psi$ centered at $x$ and $\tau_v$ denotes the translation of vector $v$.

\noindent The point $x$ is fixed for the isotopy $(H_t)_t$. Denote as $U,V$ the functions defined in \eqref{linking_notation_modify1} and \eqref{torsion_notation_modify} with respect to $(H_t)_t$, that is
\begin{equation*}
U:[0,1]\times\R\rightarrow\mathbb{T}
\end{equation*}
\begin{eqnarray}\label{linking_modifiediso1}
(s,t)&&\mapsto\theta\left((1,0),H_t(z(s))-H_t(x)\right)=\theta\left((1,0),H_t(z(s))-x\right) \quad\text{ for $s\neq 0$},\\\label{linking_modifiediso2}
(0,t) &&\mapsto \theta\left((1,0),DH_t(x)\xi\right)
\end{eqnarray}
and
\begin{equation*}
V:[0,1]\times\R\rightarrow\mathbb{T}
\end{equation*}
\begin{equation}\label{torsion_modifiediso}
\qquad (s,t)\mapsto\theta\left((1,0),DH_t(z(s))\xi\right)
\end{equation}
where $\theta$ denotes the oriented angle between the two vectors.

\noindent Observe that $U,V$ are continuous and that, for any $t$, $U(0,t)=V(0,t)$.


\noindent Define then the quantities $\tilde{U},\tilde{V}$ from $\tilde{u},\tilde{v}$ as:
\begin{equation}\label{Utildenuovalift}
\tilde{U}(s,t) = \tilde{u}(s,t)-\Theta(t)-t\dfrac{\varepsilon}{2}
\end{equation}
\begin{equation}\label{Vtildenuovalift}
\tilde{V}(s,t) = \tilde{v}(s,t) -\Theta(t)-t\dfrac{\varepsilon}{2}.
\end{equation}

\noindent These functions are continuous determinations of the angle functions $U$ and $V$, respectively. 

From the definition of $\Theta$ in \eqref{def_Theta}, for every $s\in[0,1]$ and for every $t\in(0,1]$, it follows
\begin{equation*}
\tilde{V}(s,t)-\tilde{V}(s,0)\leq -t\dfrac{\varepsilon}{2}<0.
\end{equation*}
On the other hand, by hypothesis \eqref{estimate_du_torsionneg_et_linkpos}, for any $s\in[0,1]$ it holds
\begin{equation}\label{attenzione label}
\tilde{V}(s,1)-\tilde{V}(s,0)\leq \tilde{U}(1,1)-\tilde{U}(1,0)-\varepsilon.
\end{equation}
Let $S\in[0,1]$ be a point at which the \emph{maximum} $\Theta(1)$ is achieved (see \eqref{def_Theta}), i.e.
$$
\Theta(1)=\tilde{v}(S,1)-\tilde{v}(S,0).
$$
For such $S$ we have $\tilde{V}(S,1)-\tilde{V}(S,0)=-\frac{\varepsilon}{2}$ and \eqref{attenzione label} still holds true. Therefore
\begin{equation*}
-\dfrac{\varepsilon}{2}\leq \tilde{U}(1,1)-\tilde{U}(1,0)-\varepsilon \quad\Rightarrow\quad \tilde{U}(1,1)-\tilde{U}(1,0)\geq \dfrac{\varepsilon}{2}>0.
\end{equation*}

\noindent Hence, for any $s\in[0,1]$
\begin{equation}\label{inequality_key_newiso}
\tilde{V}(s,1)-\tilde{V}(s,0)\leq -\dfrac{\varepsilon}{2}<0<\dfrac{\varepsilon}{2}\leq \tilde{U}(1,1)-\tilde{U}(1,0).
\end{equation}
\end{proof}

\begin{notazione}
We will conserve this notation of $U,V,\tilde{U},\tilde{V}$ throughout the whole section, until the conclusion of the proof.
\end{notazione}

\subsubsection{Sign concordance of Linking and Torsion for small $s$}


\begin{lemma}
Let $\tilde{U}$ and $\tilde{V}$ be the functions introduced in \eqref{Utildenuovalift} and \eqref{Vtildenuovalift}. There exists $s_0\in(0,1)$ such that for all $s\in[0,s_0]$ it holds
\begin{equation}
\tilde{U}(s,1)-\tilde{U}(s,0)\leq -\dfrac{\varepsilon}{4}<0<\tilde{U}(1,1)-\tilde{U}(1,0).
\end{equation}
\end{lemma}

\begin{proof}
By definition of $\tilde{U},\tilde{V}$ (see \eqref{Utildenuovalift} and \eqref{Vtildenuovalift}) it holds
$$
\tilde{U}(0,1)-\tilde{U}(0,0)=\tilde{V}(0,1)-\tilde{V}(0,0).
$$
Recalling the first inequality of \eqref{inequality_key_newiso}, we have
$$
\tilde{V}(0,1)-\tilde{V}(0,0)\leq-\dfrac{\varepsilon}{2}.
$$
By the continuity of the function $s\mapsto\tilde{U}(s,1)-\tilde{U}(s,0)$, we conclude that there exists $s_0\in(0,1)$ small enough such that
$$
\tilde{U}(s,1)-\tilde{U}(s,0)\leq -\dfrac{\varepsilon}{4}
$$
for any $s\in[0,s_0]$.
\end{proof}

\subsubsection{Contradiction by using the Turning Tangent Theorem}\label{parte_centrale_dimostrazione}

To sum up, we are considering an isotopy $(H_t)_{t\in[0,1]}$ in $\text{Diff}^1(\R^2)$ such that:
\begin{itemize}
\item $H_0=Id_{\R^2}$ and $H_1=H$;
\item the point $x\in\R^2$ is fixed with respect to $(H_t)_{t\in[0,1]}$;
\item for any $s\in[0,1]$,
\begin{equation}\label{cidicequalcheipotesi}
\tilde{V}(s,1)-\tilde{V}(s,0)<0<\tilde{U}(1,1)-\tilde{U}(1,0);
\end{equation}
\item for any $s<s_0$,
\begin{equation}\label{assicura_esistenza_sbar}
\tilde{U}(s,1)-\tilde{U}(s,0)<-\dfrac{\varepsilon}{4}<0<\tilde{U}(1,1)-\tilde{U}(1,0).
\end{equation}
\end{itemize}

\noindent By eventually changing the reference system on the plane, assume that $x$ is the origin and that the first vector of the canonical basis coincides with $\xi=y-x$.\\
\noindent Denote
\begin{equation}\label{definizione_sbar}
\bar{s}:= \min_{s\in(0,1)}\{s :\ \tilde{U}(s,1)-\tilde{U}(s,0)=0\}.
\end{equation}
The corresponding $z(\bar{s})\in[x,y]$ is the first point of the segment for which the lift of the angle associated to $H_1(z(s))$ is zero, i.e. $\tilde{U}(\bar{s},1)-\tilde{U}(\bar{s},0)=0$.\\
\noindent Such $\bar{s}$ exists by inequality \eqref{assicura_esistenza_sbar} and by continuity of $\tilde{U}$.\\
\noindent Recall that $\tilde{U}(s,1)-\tilde{U}(s,0)$ does not depend on the chosen lift. It is important considering $\bar{s}$ as the \emph{first} point of intersection of the image of the segment at time $t=1$ with the first coordinate axis (which is the segment at time $t=0$). Otherwise we could have no control on the image of the tangent vector through the isotopy.

The proof is divided into 3 cases: starting with the simpler one, we then move on to the most general case.

\begin{itemize}

\begin{figure}[h]
\centering
\includegraphics[scale=.25]{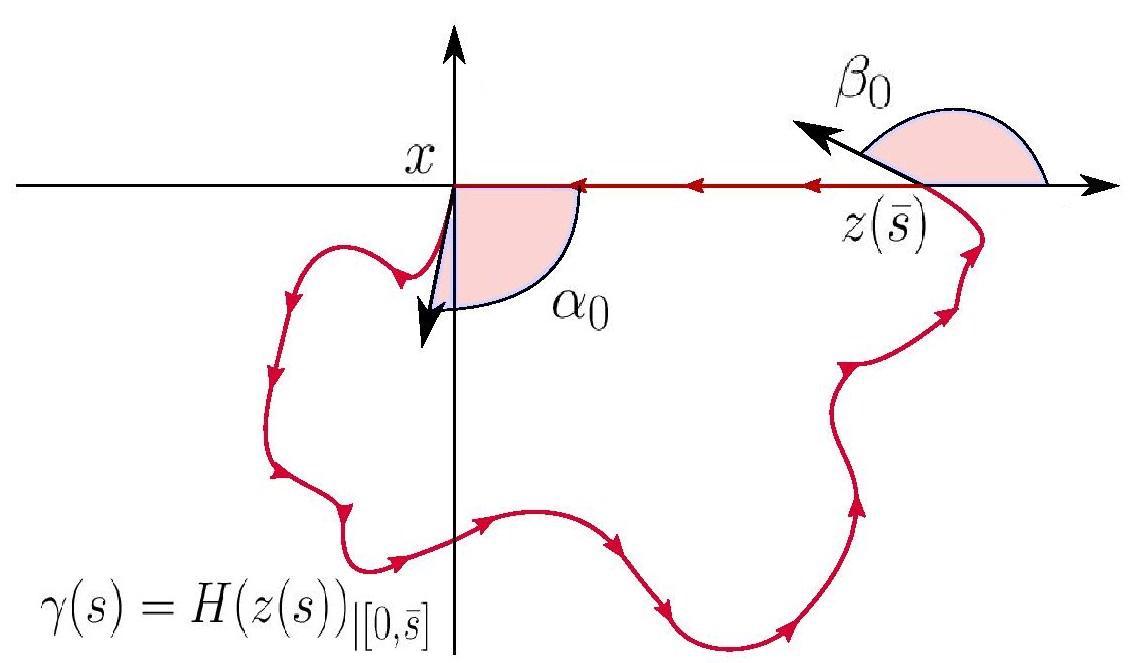}
\caption{The first case.}
\label{caso1}
\end{figure}

\item[\emph{First case:}] As a first simpler case, consider the situation presented in Figure \eqref{caso1}. That is to say, suppose that
\begin{itemize}
\item $\alpha_0=\tilde{U}(0,1)-\tilde{U}(0,0)=\tilde{V}(0,1)-\tilde{V}(0,0)\in(-2\pi,0)$;
\item the curve made up of $H(z(s))_{\lvert [0,\bar{s}]}$ and the segment $[z(\bar{s}),x]$ is a simple, closed, piecewise regular, parametrized curve.
\end{itemize}

Denote
$$
\gamma(s)=H(z(s))_{s\in[0,\bar{s}]}.
$$
\noindent According to this notation, the quantity $\tilde{V}(s,1)-\tilde{V}(s,0)$ is a measure of the angle between the first coordinate axis direction vector and $\dot{\gamma}(s)$.\\
\noindent By hypothesis -the first inequality of \eqref{cidicequalcheipotesi}- for any $s\in[0,1]$ we have
\begin{equation}\label{inequality 64 nella versione vecchia}
\tilde{V}(s,1)-\tilde{V}(s,0)<0.
\end{equation}

\indent The angle $V(\bar{s},1)-V(\bar{s},0)$ admits a measure $\beta_0\in[0,\pi]$. Indeed, in a neighborhood of $z(\bar{s})$, for $s<\bar{s}$, the curve $\gamma(s)$ crosses the first coordinate axis from the bottom up. So, the tangent vector $\dot{\gamma}(\bar{s})$ has a non negative second coordinate and lies in the upper half-plane.\\
\noindent Look at the continuous determination $\tilde{V}(\bar{s},1)-\tilde{V}(\bar{s},0)$: we have
\begin{equation}
\tilde{V}(\bar{s},1)-\tilde{V}(\bar{s},0)=\beta_0 +2\pi k\qquad k\in\Z.
\end{equation}

\noindent By inequality \eqref{inequality 64 nella versione vecchia}, necessarily
\begin{equation}\label{condition_k_negativo}
k\leq -1.
\end{equation}

\indent Since the curve made up of $\gamma(s)$ and $[z(\bar{s}),x]$ is simple, closed and piecewise regular, we can apply the Turning Tangent Theorem on it (see Chapter 4, Section 5 in \cite{docarmo}). We obtain
\begin{equation*}
\begin{split}
\left( (\tilde{V}(\bar{s},1)-\tilde{V}(\bar{s},0)) - ( \tilde{V}(0,1)-\tilde{V}(0,0) ) \right) +  \\
+ \left( \pi - \beta_0 \right) + \left( \alpha_0+\pi \right) = 2\pi
\end{split}
\end{equation*}
that is
\begin{equation*}
\beta_0+2\pi k -\alpha_0 +\pi-\beta_0 +\alpha_0 +\pi = 2\pi(1+k)=2\pi.
\end{equation*}
This last equality implies $k=0$ and contradicts \eqref{condition_k_negativo}.

\begin{figure}[h]
\centering
\includegraphics[scale=.25]{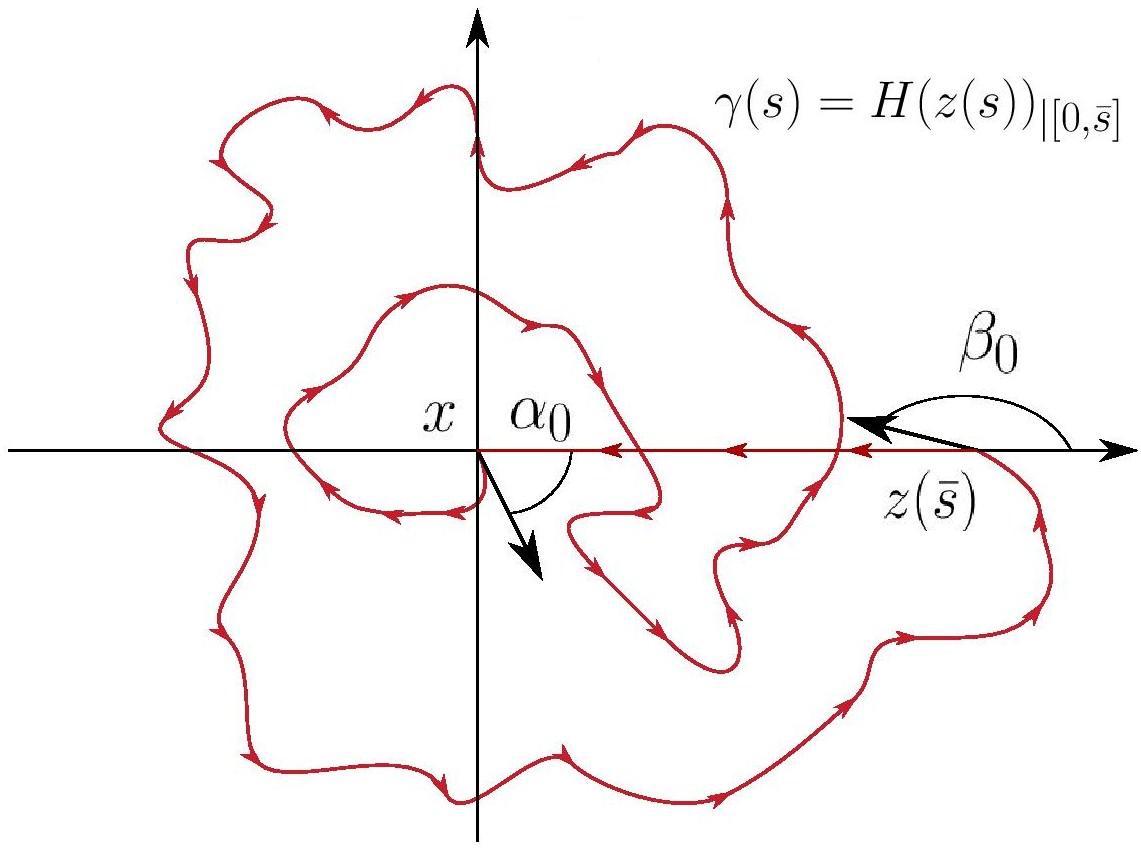}
\caption{The second case.}
\label{caso2}
\end{figure}

\item[\emph{Second case:}] Consider the case presented in Figure \eqref{caso2}. We allow the curve made up of $\gamma(s):=H(z(s))_{\lvert[0,\bar{s}]}$ and the segment $[z(\bar{s}),x]$ to have self-intersections, but we require some regularity conditions at the origin.\footnote{These conditions will be precised later.}

Define the function $\Gamma:[0,\bar{s}]\rightarrow\R^+\times\R$ as
\begin{equation}
\Gamma(s)=\left(\Gamma_1(s),\Gamma_2(s)\right) = \left( r(s), \tilde{U}(s,1)-\tilde{U}(s,0)\right)
\end{equation}
where $r(s) = \norma{H(z(s))-x}\in\R^+$.\\
\noindent Denote
\begin{equation*}
\begin{split}
P: &\ \R^+\times\R\rightarrow\R^2 \\
& (r,\theta)\mapsto (r\,\cos\theta, r\,\sin\theta).
\end{split}
\end{equation*}
Notice that $P_{\vert\R^+_*\times\R}$ is the universal covering of $\R^2\setminus\{(0,0)\}$. Since $P\circ\Gamma=\gamma$, then $\Gamma$ is a lift of $\gamma$ through $P$.
\noindent Identifying the plane $\R^2$ with the complex one $\C$, we have
$$
\gamma(s)=\Gamma_1(s) e^{i\Gamma_2(s)}.
$$
In other words, $(\Gamma_1(s),\Gamma_2(s))$ provide some polar ``coordinates''.

\noindent By hypothesis \eqref{assicura_esistenza_sbar} and by definition of $\bar{s}$ in \eqref{definizione_sbar}, it holds
\begin{equation}
\tilde{U}(s,1)-\tilde{U}(s,0)=\Gamma_2(s)\leq 0\qquad\forall s\in[0,\bar{s}].
\end{equation}

\noindent Therefore, the curve $\Gamma$ lies on the low quarter of the half-plane $\R^+\times\R$.
\noindent Precisely
\begin{align*}
\Gamma_1(s)>0,\qquad\Gamma_2(s)<0\qquad\forall s\in(0,\bar{s}), \\
\Gamma_1(0)=0,\qquad \Gamma_2(0)< 0, \\
\Gamma_1(\bar{s})=\norma{z(\bar{s})-x},\qquad \Gamma_2(\bar{s})=0.
\end{align*}

\begin{assumption}\label{ipotesi in piu secondo caso}
Throughout this second case, assume that $\Gamma$ is sufficiently regular at the origin, that is there exists
\begin{equation*}
\dot{\Gamma}(0) := \lim_{s\rightarrow 0^+} \dot{\Gamma}(s)\neq 0.
\end{equation*}
\end{assumption}

\begin{figure}[h]
\centering
\includegraphics[scale=.23]{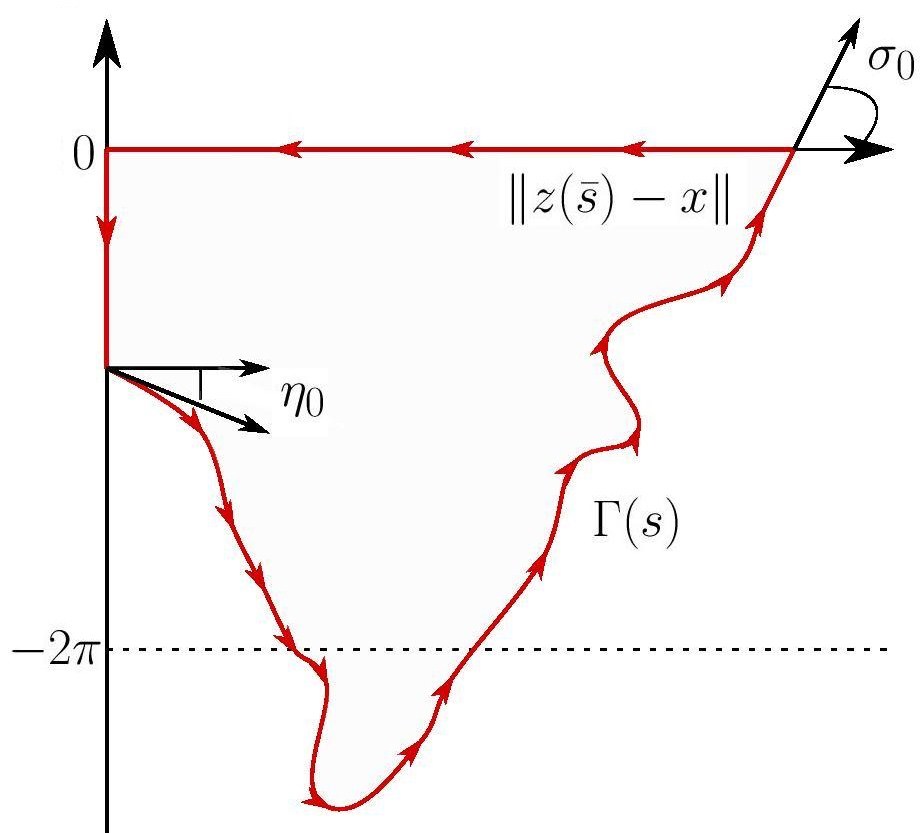}
\caption{The function $\Gamma(s)$ in the second case.}
\label{aggiungicitfig1}
\end{figure}

\begin{notazione}\label{notazione intro curva caso 2}
Consider the curve in $\R^+\times\R$, made up of
\begin{itemize}
\item[$(i)$] $\Gamma(s)$ for $s\in[0,\bar{s}]$;
\item[$(ii)$] the horizontal segment $\{0\}\times[0,r(\bar{s})]$, followed with decreasing radius;
\item[$(iii)$] the vertical segment $[\tilde{U}(0,1)-\tilde{U}(0,0),0]\times\{0\}$, followed downward.
\end{itemize}
This curve, thanks to Assumption \ref{ipotesi in piu secondo caso} and thanks to the definition of $\bar{s}$ in \eqref{definizione_sbar}, is a simple, closed, piecewise regular curve (see Figure \ref{aggiungicitfig1}).
\end{notazione}


The vector $\dot{\Gamma}(0)$ is oriented to the right in the plane $\R^+\times\R$. Hence, the angle between the first coordinate axis direction vector and $\dot{\Gamma}(0)$ admits a measure $\eta_0\in[-\frac{\pi}{2},\frac{\pi}{2}]$.\\
\noindent Denote as $\sigma_0$ the measure contained in the interval $[0,\pi]$ of the angle between the first coordinate axis direction vector and $\dot{\Gamma}(\bar{s})$. Such a measure exists since in a neighborhood of $\Gamma(\bar{s})$ the curve $\Gamma$ crosses the first coordinate axis from the bottom up and so the tangent vector $\dot{\Gamma}(\bar{s})$ lies then in the upper half-plane. 

\begin{notazione}\label{denotiamo gli angoli e i lift di Gamma}
Denote as
\begin{equation*}
\prec(\dot{\Gamma}):[0,\bar{s}]\rightarrow\mathbb{T}
\end{equation*}
the oriented angle function between the first coordinate axis direction vector and the vector $\dot{\Gamma}(s)$.\\
The notation $\tilde{\prec}(\dot{\Gamma}):[0,\bar{s}]\rightarrow\R$ refers to the continuous determination of the angle function $\prec(\dot{\Gamma})$ such that $\tilde{\prec}(\dot{\Gamma}(0))=\eta_0\in [-\frac{\pi}{2},\frac{\pi}{2}]$.
\end{notazione}

\noindent Since $\sigma_0$ and $\tilde{\prec}(\dot{\Gamma}(\bar{s}))$ are lifts of the same oriented angle, we have
$$
\tilde{\prec}(\dot{\Gamma}(\bar{s}))=\sigma_0+2\pi j\qquad j\in\Z.
$$
Apply now the Turning Tangent Theorem to the closed curve highlighted in Notation \ref{notazione intro curva caso 2}. We obtain
\begin{equation*}
(\sigma_0+2\pi j-\eta_0)+(\pi-\sigma_0)+(\dfrac{\pi}{2})+(\eta_0+\dfrac{\pi}{2})=2\pi
\end{equation*}
and so
\begin{equation}\label{lajdeveesserenulla}
2\pi(1+j)=2\pi\qquad\Leftrightarrow\qquad j=0.
\end{equation}
Hence
\begin{equation}\label{ultima etichetta nuova}
\tilde{\prec}(\dot{\Gamma}(\bar{s}))=\sigma_0.
\end{equation}

Let us look now at the relation between the tangent vectors of $\Gamma(\cdot)$ and the tangent ones of $\gamma(\cdot)=H(z(\cdot))$.\\
\noindent By hypothesis \eqref{cidicequalcheipotesi}, it holds
\begin{equation}\label{hypoth_torsion}
\tilde{V}(s,1)-\tilde{V}(s,0)<0\qquad\forall s\in[0,\bar{s}]
\end{equation}
where $\tilde{V}(s,1)-\tilde{V}(s,0)$ is a continuous determination of the angle function between the first coordinate axis direction vector and $\dot{\gamma}(s)$.\\
\noindent Denote
\begin{equation}\label{quelque info sur Vtilde}
\tilde{V}(\bar{s},1)-\tilde{V}(\bar{s},0)=\beta_0+2\pi k,k\in\Z
\end{equation}
where $\beta_0$ is the measure of the angle $V(\bar{s},1)-V(\bar{s},0)$ in $[0,\pi]$. Such a measure exists since in a neighborhood of $z(\bar{s})$ the curve $\gamma$ crosses the first coordinate axis from the bottom up. Hence, the vector $\dot{\gamma}(\bar{s})$ has non negative second coordinate.\\
\noindent From \eqref{hypoth_torsion}, it holds then
\begin{equation}\label{condizionesullak}
k\leq -1.
\end{equation}

We need now the following:
\begin{proposizione}\label{stimaangoli}
Let $I\subset \R$ be an interval and let $M,N$ be two $2$-dimensional oriented Riemannian manifolds. Denote the tangent projections as $\pi_M:TM\rightarrow M$, $\pi_N:TN\rightarrow N$. 
\noindent Let $f:M\rightarrow N$ be a local diffeomorphism which preserves the orientation and let $J_1:I\rightarrow TM$, $J_2:I\rightarrow TM$ be continuous functions such that
\begin{equation}\label{stanno nello stesso spazio tangente}
\pi_M\circ J_1=\pi_M\circ J_2.
\end{equation}
Suppose that, for any $t\in I$, $J_i(t)\neq 0$, $i=1,2$ and let $\theta:I\rightarrow\R$ be a continuous determination of the angle function between the image vectors $J_1,J_2$.

\noindent Then, there exists a continuous determination $\Theta:I\rightarrow\R$ of the angle function between the image vectors $Df\circ J_1, Df\circ J_2$ such that
\begin{equation}
\abs{\theta(s)-\Theta(s)}<\pi\qquad\forall s\in I
\end{equation}
\end{proposizione}

\noindent We postpone the proof of this Proposition to Appendix \ref{appendice prova prop}.

\noindent We apply Proposition \ref{stimaangoli} to $I=(0,\bar{s}]\subset\R$, $M=(\R^+\setminus\{0\})\times\R$, $N=\R^2\setminus\{0\}$ and
$$
f:(\R^+\setminus\{0\})\times\R\rightarrow\R^2\setminus\{0\}
$$
$$
(r,\theta)\mapsto(r\cos(\theta),r\sin(\theta)).
$$
Observe that the determinant of $Df(r,\theta)$ is equal to $r$ and so always positive: this assures us that $f$ is a local diffeomorphism which preserves the orientation.
\noindent Consider
\begin{equation*}
\begin{split}
J_1:(0,\bar{s}]\rightarrow TM&=(\R^+\setminus\{0\})\times\R\times\R^2\\
s\mapsto&(\Gamma(s),(1,0))
\end{split}
\end{equation*}
and
\begin{equation*}
\begin{split}
J_2:(0,\bar{s}]\rightarrow TM&=(\R^+\setminus\{0\})\times\R\times\R^2\\
s\mapsto&(\Gamma(s),\dot{\Gamma}(s)).
\end{split}
\end{equation*}
Then
\begin{equation*}
\begin{split}
Df\circ J_1&:(0,\bar{s}]\rightarrow TN=(\R^2\setminus\{0\})\times\R^2\\
s&\mapsto\left(\gamma(s),\dfrac{\gamma(s)}{\norma{\gamma(s)}}\right)
\end{split}
\end{equation*}
and
\begin{equation*}
\begin{split}
Df\circ J_2&:(0,\bar{s}]\rightarrow TN=(\R^2\setminus\{0\})\times\R^2\\
s&\mapsto(\gamma(s),\dot{\gamma}(s)).
\end{split}
\end{equation*}
The function $\tilde{\prec}(\dot{\Gamma})$ introduced in Notation \ref{denotiamo gli angoli e i lift di Gamma} is a continuous determination of the angle function between $J_1(s)$ and $J_2(s)$. By Assumption \ref{ipotesi in piu secondo caso}, the function $\tilde{\prec}(\dot{\Gamma})$ is continuous at $s=0$.\\
\noindent Remind that, by our choice
\begin{equation}\label{AAA}
\tilde{\prec}(\dot{\Gamma}(0))=\eta_0\in\left[-\dfrac{\pi}{2},\dfrac{\pi}{2}\right].
\end{equation}
\noindent Observe that $s\mapsto\left(\tilde{V}(s,1)-\tilde{V}(s,0)\right) - \left(\tilde{U}(s,1)-\tilde{U}(s,0)\right)$ is a continuous determination of the angle function between $Df\circ J_1(s)$ and $Df\circ J_2(s)$. By our choice of $\tilde{V},\tilde{U}$, for any $t$ we have $\tilde{V}(0,t)=\tilde{U}(0,t)$ and in particular
\begin{equation}\label{BBB}
\left(\tilde{V}(0,1)-\tilde{V}(0,0)\right) - \left( \tilde{U}(0,1)-\tilde{U}(0,0)\right) =0.
\end{equation}
\noindent From \eqref{AAA}, \eqref{BBB} and the continuity of the involved functions, there exists $S>0$ small enough such that
$$
\abs{ \tilde{\prec}(\dot{\Gamma}(S)) - \left( \left( \tilde{V}(S,1)-\tilde{V}(S,0) \right) - \left( \tilde{U}(S,1)-\tilde{U}(S,0) \right) \right) }<\pi.
$$
By Proposition \ref{stimaangoli}, we deduce that for any $s\in(0,\bar{s}]$
$$
\abs{ \tilde{\prec}(\dot{\Gamma}(s)) - \left( \left( \tilde{V}(s,1)-\tilde{V}(s,0) \right) - \left( \tilde{U}(s,1)-\tilde{U}(s,0) \right) \right) }<\pi.
$$
\noindent In particular, at $s=\bar{s}$ by \eqref{ultima etichetta nuova}, \eqref{quelque info sur Vtilde} and \eqref{definizione_sbar}
\begin{equation}\label{stellina}
\abs{\tilde{\prec}(\dot{\Gamma}(\bar{s}))-\left( (\tilde{V}(\bar{s},1)-\tilde{V}(\bar{s},0))-(\tilde{U}(\bar{s},1)-\tilde{U}(\bar{s},0)) \right)}=\abs{ \sigma_0 -\beta_0 - 2\pi k }<\pi.
\end{equation}

\begin{claim}\label{quantobellstoclaim}
The quantity $\sigma_0-\beta_0$ is in the open interval $(-\pi,\pi)$.
\end{claim}
\textit{Proof of the claim.} Because $\sigma_0,\beta_0\in[0,\pi]$, the difference $\sigma_0-\beta_0$ is in $[-\pi,\pi]$. Arguing by contradiction, suppose that $\sigma_0-\beta_0=\pi$, that is $\sigma_0=\pi,\beta_0=0$. The measure $\sigma_0$ is a lift of the angle between $(1,0)$ and $\dot{\Gamma}(\bar{s})$, while $\beta_0$ is a lift of the angle between $\gamma(\bar{s})/\norma{\gamma(\bar{s})}$ and $\dot{\gamma}(\bar{s})$, which are the vectors $Df(\Gamma(\bar{s}))(1,0)$ and $Df(\Gamma(\bar{s}))\dot{\Gamma}(\bar{s})$. Since $Df(\Gamma(\bar{s}))$ is a linear function and by inequality \eqref{stellina}, this case cannot occur. Similarly the case $\sigma_0-\beta_0=-\pi$ is excluded.
\hfill\qed

Since $\sigma_0-\beta_0\in(-\pi,\pi)$ and by \eqref{stellina}, we deduce that $k=0$. This inequality contradicts condition \eqref{condizionesullak} and we conclude.


\item[\emph{Third case:}] Finally, consider the most general case, presented in Figure \eqref{caso3}. We allow now the vector $\dot{\Gamma}(0)$ not to exist or to be null.

\begin{figure}[h]
\centering
\includegraphics[scale=.3]{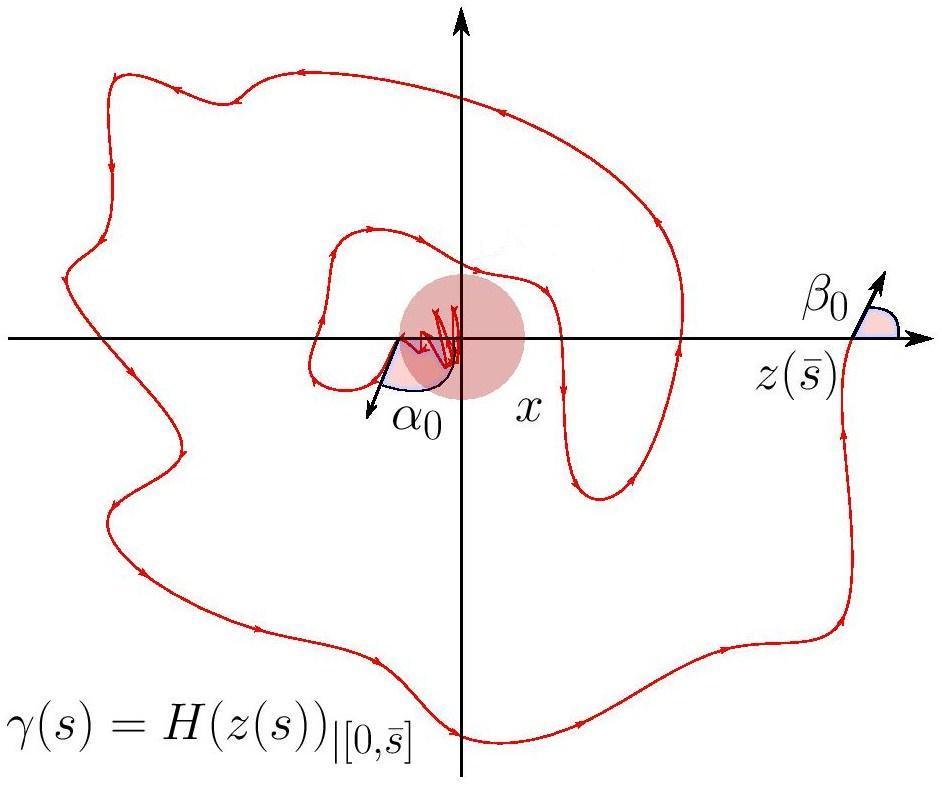}
\caption{The most general case.}
\label{caso3}
\end{figure}

The Turning Tangent Theorem can no more be applied on the curve used in the second case.\\
\noindent Fix $\rho\in (0,\Gamma_1(\bar{s}))$ and consider the vertical line $r\equiv\rho$ in $\R^+\times\R$.

\noindent The notations $\Gamma_1(\cdot),\Gamma_2(\cdot)$ refer to the first and second coordinates, respectively, of the curve $\Gamma$ in $\R^+\times\R$. Define then
\begin{equation}
s_{\rho}:=\max_{s\in[0,\bar{s}]} \{ s :\ \Gamma_1(s)=\rho \}.
\end{equation}
This is a \emph{maximum} since $\Gamma_1$ is a continuous function considered on a compact interval $[0,\bar{s}]$ where $\Gamma_1(0)=0$ and $\Gamma_1(\bar{s})>\rho$.

\noindent Observe that
\begin{equation*}
\lim_{\rho\rightarrow 0}s_{\rho}=0
\end{equation*}
by the continuity of the function $\Gamma_1(\cdot)$, the compactness of the interval involved and the fact that $s=0$ is the only point for which the first coordinate projection of the curve vanishes.



Denote as $\eta_0$ the measure of the angle between the first coordinate axis direction vector and $\dot{\Gamma}(s_{\rho})$ in the interval $[-\frac{\pi}{2},\frac{\pi}{2}]$. This choice is possible since by definition of $s_{\rho}$ the vector $\dot{\Gamma}(s_{\rho})$ is oriented to the right.\\
\noindent Let
$$
\prec(\dot{\Gamma}):[s_{\rho},\bar{s}]\rightarrow\mathbb{T}
$$
denote the oriented angle between the first coordinate axis direction vector and the vector $\dot{\Gamma}(s)$ and denote
$$
\tilde{\prec}(\dot{\Gamma}):[s_{\rho},\bar{s}]\rightarrow\R
$$
the continuous determination of the angle function such that $\tilde{\prec}(\dot{\Gamma}(s_{\rho}))=\eta_0\in[-\frac{\pi}{2},\frac{\pi}{2}]$.\\

\begin{figure}[h]
\centering
\includegraphics[scale=.25]{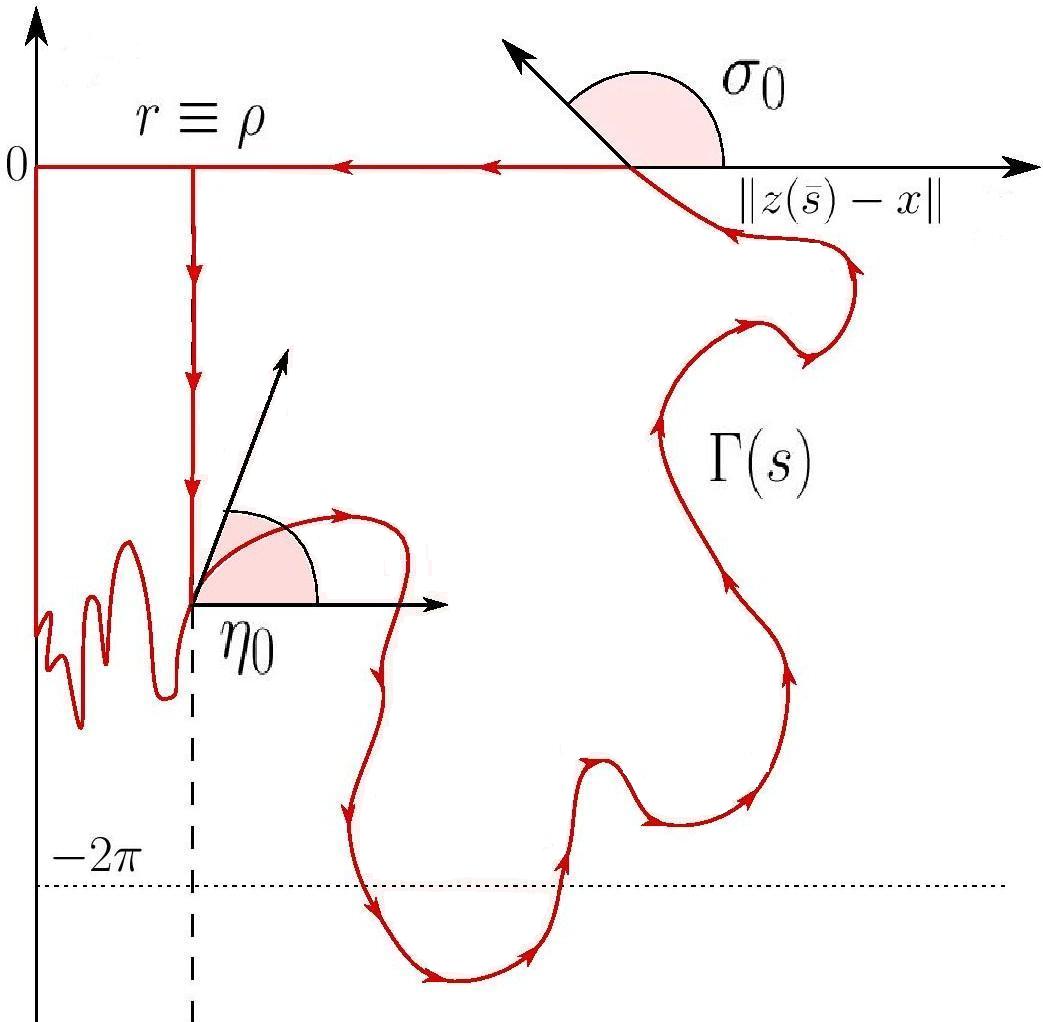}
\caption{The function $\Gamma(s)$ in the general case.}
\label{aggiungicitfig2}
\end{figure}

\begin{claim}
If $\rho$ is small enough, for any $s\in[s_{\rho},\bar{s}]$
\begin{equation}\label{stimaangleesserho}
\abs{\tilde{\prec}(\dot{\Gamma}(s)) - \left( \left( \tilde{V}(s,1)-\tilde{V}(s,0) \right) - \left( \tilde{U}(s,1)-\tilde{U}(s,0) \right) \right)}<\pi.
\end{equation}
\end{claim}

\textit{Proof of the claim.} Recall that $\lim_{\rho\rightarrow 0^+}s_{\rho}=0$ and functions $\tilde{V},\tilde{U}$ are continuous. Moreover, since $\tilde{V}(0,t)=\tilde{U}(0,t)$ for any $t$, $\left( \tilde{V}(0,1)-\tilde{V}(0,0) \right) - \left( \tilde{U}(0,1)-\tilde{U}(0,0) \right)=0$.\\
\noindent Then, for any $\varepsilon>0$, there exists $\rho>0$ small enough such that
\begin{equation*}
\abs{\left(\tilde{V}(s_{\rho},1)-\tilde{V}(s_{\rho},0) \right) - \left( \tilde{U}(s_{\rho},1)-\tilde{U}(s_{\rho},0) \right)}<\varepsilon.
\end{equation*}

\noindent So, it holds
\begin{equation*}
\abs{\tilde{\prec}(\dot{\Gamma}(s_{\rho})) - \left( \left( \tilde{V}(s_{\rho},1)-\tilde{V}(s_{\rho},0) \right) - \left( \tilde{U}(s_{\rho},1)-\tilde{U}(s_{\rho},0) \right) \right)}<\dfrac{\pi}{2}+\varepsilon.
\end{equation*}
By selecting $\varepsilon>0$ small enough such that $\frac{\pi}{2}+\varepsilon<\pi$, we have
$$
\abs{\tilde{\prec}(\dot{\Gamma}(s_{\rho})) - \left( \left( \tilde{V}(s_{\rho},1)-\tilde{V}(s_{\rho},0) \right) - \left( \tilde{U}(s_{\rho},1)-\tilde{U}(s_{\rho},0) \right) \right)}<\pi.
$$
By applying Proposition \ref{stimaangoli}, inequality \eqref{stimaangleesserho} holds for any $s\in[s_{\rho},\bar{s}]$.\hfill\qed

Denote as $\sigma_0$ the measure contained in $[0,\pi]$ of the angle $\prec(\dot{\Gamma}(\bar{s}))$: again, this is possible because in a neighborhood of $\Gamma(\bar{s})$, the curve $\Gamma$ crosses the first coordinate axis from the bottom up.\\
\noindent Let $\beta_0$ be the measure of the angle $V(\bar{s},1)-V(\bar{s},0)$ contained in $[0,\pi]$.\\
\noindent Since $\sigma_0$ and $\tilde{\prec}(\dot{\Gamma}(\bar{s}))$ are continuous lifts of the angle $\prec(\dot{\Gamma}(\bar{s}))$, we have
$$
\tilde{\prec}(\dot{\Gamma}(\bar{s}))=\sigma_0+2\pi l\qquad l\in\Z,
$$
$$
\tilde{V}(\bar{s},1)-\tilde{V}(\bar{s},0) =\footnote{By definition of $\bar{s}$, $\tilde{U}(\bar{s},1)-\tilde{U}(\bar{s},0)$ is null.}\left( \tilde{V}(\bar{s},1)-\tilde{V}(\bar{s},0) \right)-\left( \tilde{U}(\bar{s},1)-\tilde{U}(\bar{s},0) \right)=
$$
$$
=\beta_0+2\pi j\qquad j\in\Z.
$$
By inequality \eqref{stimaangleesserho} it holds
\begin{equation}
\abs{\tilde{\prec}(\dot{\Gamma}(\bar{s})) - \left( (\tilde{V}(\bar{s},1)-\tilde{V}(\bar{s},0)) - (\tilde{U}(\bar{s},1)-\tilde{U}(\bar{s},0)) \right)}=\abs{\sigma_0+2\pi l -\beta_0 -2\pi j}<\pi.
\end{equation}
By hypothesis \eqref{cidicequalcheipotesi}, $j\leq -1$.\\
\begin{claim}
The quantity $\sigma_0-\beta_0$ is in the open interval $(-\pi,\pi)$.
\end{claim}
The argument is the same as Claim \ref{quantobellstoclaim} in the second case.

\noindent Therefore $l=j$ and so
\begin{equation}\label{sottoilclaim}
l\leq -1.
\end{equation}

\noindent Let us now consider the curve made up of
\begin{itemize}
\item[$(i)$] $\Gamma_{\lvert[s_{\rho},\bar{s}]}$, positively oriented;
\item[$(ii)$] the horizontal segment $\{0\}\times\{r :\ \rho\leq r\leq\norma{z(\bar{s})-x}\}$, followed with decreasing radius;
\item[$(iii)$] the vertical segment $[\Gamma_2(s_{\rho}),0]\times \{ r\equiv\rho \}$, followed downward.
\end{itemize}

\noindent This curve is a simple, closed, piecewise regular, parametrized one thanks to the regularity of the polar coordinates away from the origin and to the absence of self-intersections by the definition of $s_{\rho}$ (see Figure \ref{aggiungicitfig2}).\\
\noindent Apply the Turning Tangent Theorem to this curve. We obtain then
\begin{equation*}
\left( \sigma_0+2\pi l - \eta_0 \right) + (\pi-\sigma_0) + \dfrac{\pi}{2} + (\eta_0+\dfrac{\pi}{2})=2\pi
\end{equation*}

\noindent that is
\begin{equation*}
2\pi(1+l)=2\pi.
\end{equation*}
This implies $l=0$, contradicting inequality \eqref{sottoilclaim}.
\end{itemize}
\hfill\qed

\section{Results over the Torsion and the Linking number for a twist map}\label{torsion twist map}
For the following definition we refer to \cite{lecalvezproprietes} and \cite{crovisier}. In addition, other interesting references are \cite{matherexistence}, \cite{mathervariational} and \cite{matherglancing}.
\begin{definizione}\label{twist map alla crovisier}
A \emph{positive twist map} (resp. \emph{negative}) $f:\mathbb{A}\rightarrow\mathbb{A}$ is a $\mathcal{C}^1$ diffeomorphism isotopic to the identity such that for any lift $F:\R^2\rightarrow\R^2$ and for any $x\in\R$ the function
\begin{equation}\label{crovisier 1 crois}
\R\ni y\mapsto p_1\circ F(x,y)\in\R
\end{equation}
is a strictly increasing (resp. decreasing) diffeomorphism.
\end{definizione}

\begin{remark}
All over the literature (see \cite{lecalvezproprietes} and \cite{crovisier}), the definition of \emph{positive twist map} asks also the further condition that for any lift $F:\R^2\rightarrow\R^2$ and for any $x\in\R$ the function
\begin{equation}\label{crovisier 2 decrois}
\R\ni y\mapsto p_1\circ F^{-1}(x,y)\in\R
\end{equation}
is a decreasing diffeomorphism of $\R$.\\
\noindent Actually, Definition \ref{twist map alla crovisier} implies this condition and we omit it.
\end{remark}
\begin{remark}
In the sequel we work with positive twist maps. The results we are going to present hold, properly adapted, also for negative twist maps.
\end{remark}

\begin{remark}\label{remark isotopia le calvez}
The torsion at any point for a positive twist map $f$ is independent from the choice of the isotopy $(f_t)_t$, thanks to Proposition \ref{indipendenza torsione isotopia}.\\
\noindent In Section 2 of \cite{lecalvezproprietes}, Patrice Le Calvez proved that any positive twist map $f:\mathbb{A}\rightarrow\mathbb{A}$ can be joined to the identity $Id_{\mathbb{A}}$ through an isotopy $(f_t)_{t\in[0,1]}$ in $\text{Diff}^1(\mathbb{A})$ such that $f_0=Id_{\mathbb{A}}$, $f_1=f$ and for any $t\in(0,1]$ each $f_t$ is a positive twist map.
\end{remark}
\begin{notazione}
\noindent In the following, the annulus $\mathbb{A}$ is endowed with the standard Riemannian metric and trivialization. We fix the counterclockwise orientation and consider as reference vector field the constant one $(0,1)$. Caution! We emphasize that we are changing the reference vector field with respect to the previous sections.
\end{notazione}

\subsection{Properties of Torsion for twist maps}
\noindent We now prove Theorem \ref{thm intro 2} presented in the introduction. The proof requires some preliminary steps.
\begin{proposizione}\label{main thm torsion}
Let $f:\mathbb{A}\rightarrow\mathbb{A}$ be a positive twist map. For any $\bar{z}\in\mathbb{A}$ it holds
\begin{equation}
Torsion_1((f_t)_t,\bar{z},(0,1))\in(-\pi,0).
\end{equation}
\end{proposizione}
\begin{proof}
From Proposition \ref{indipendenza torsione isotopia}, the torsion does not depend from the choice of the isotopy. Therefore, we use the isotopy given by P.~Le Calvez (see Remark \ref{remark isotopia le calvez}): for any $t\in(0,1]$ the ${\cal{C}}^1$ diffeomorphism $f_t$ is a positive twist map.\\
\noindent Let $(F_t)_t$ be the lifted isotopy of $(f_t)_t$ such that $F_0=Id_{\R^2}$. It joins the identity to $F_1=F$, a lift of $f$. The point $z=(x,y)\in\R^2$ denotes a lift of the point $\bar{z}\in\mathbb{A}$.\\
\noindent Look then at
$$
Torsion_1((F_t)_t,z,(0,1)).
$$
It is the variation of a continuous determination $\tilde{v}((F_t)_t)(z,(0,1),\cdot)$ of the oriented angle function between $(0,1)$ and $DF_t(z)(0,1)$. Recall that it is independent from the choice of the continuous determination of the angle function (see Proposition \ref{proprieta torsione}).\\
\noindent By the choice of the isotopy, for any $t\in(0,1]$, $f_t$ is a positive twist map. Then, since $F_t$ is a lift of $f_t$, for any $x\in\R$ the function
$$
\R\ni y\mapsto p_1\circ F_t(x,y)\in\R
$$
is an increasing diffeomorphism of $\R$. In particular, its derivative is always positive, that is
\begin{equation}
D(p_1\circ F_t)(z)\begin{pmatrix}
0 \\ 1
\end{pmatrix}>0.
\end{equation}
\noindent For any $t\in(0,1]$ the first component of the image vector $DF_t(z)(0,1)$ is positive. The vector remains in the right half-plane and it cannot cross the vertical anymore. Thus, the variation
$$
\tilde{v}((F_t)_t)(z,(0,1),t)-\tilde{v}((F_t)_t)(z,(0,1),0)
$$
has to stay in the interval $(-\pi,0)$ for any $t\in(0,1]$, thanks also to the continuity of the lift. We then conclude that
\begin{equation}
\tilde{v}((F_t)_t)(z,(0,1),1)-\tilde{v}((F_t)_t)(z,(0,1),0) = Torsion_1((F_t)_t,z,(0,1))\in(-\pi,0).
\end{equation}
\end{proof}

\begin{proposizione}\label{prop torsion twist nimportequevecteur}
Let $f:\mathbb{A}\rightarrow\mathbb{A}$ be a positive twist map. Let $\bar{z}\in\mathbb{A}$ and let $\xi\in T_{\bar{z}}\mathbb{A}\setminus\{0\}$. Then it holds
\begin{equation}
Torsion_1((f_t)_t,\bar{z},\xi)\in(-2\pi,\pi).
\end{equation}
\end{proposizione}
\begin{proof}
We use the notations of Proposition \ref{PROP IMPO vV}. Then $W(0,\cdot)$ and $W(-\pi,\cdot)$ are continuous determinations of $v((f_t)_t)(\bar{z},(0,1),\cdot)$ and $v((f_t)_t)(\bar{z},(0,-1),\cdot)$ respectively, such that $W(0,0)=0$ and $W(-\pi,0)=-\pi$.\\
\noindent We assume that $\xi$ is in the right half-plane. Then $v((f_t)_t)(\bar{z},\xi,0)$ admits a measure $\alpha\in[-\pi,0]=[W(-\pi,0),W(0,0)]$. Let us denote such a measure $\tilde{v}((f_t)_t)(\bar{z},\xi,0)$. There exists $s\in[-\pi,0]$ such that $\tilde{v}((f_t)_t)(\bar{z},\xi,\cdot)=W(s,\cdot)$ and so by point $(ii)$ of Proposition \ref{PROP IMPO vV}
$$
W(-\pi,1)\leq\tilde{v}((f_t)_t)(\bar{z},\xi,1)=W(s,1)\leq W(0,1).
$$
This implies
$$
W(-\pi,1)-W(0,0)\leq\tilde{v}((f_t)_t)(\bar{z},\xi,1)-\tilde{v}((f_t)_t)(\bar{z},\xi,0)\leq W(0,1)-W(-\pi,0).
$$
Because of $(iii)$ of Proposition \eqref{PROP IMPO vV}, we have that $W(-\pi,1)=W(0,1)-\pi$ and $W(-\pi,0)=W(0,0)-\pi$. From these equalitites and since $Torsion_1((f_t)_t,\bar{z},\xi)=\tilde{v}((f_t)_t)(\bar{z},\xi,1)-\bar{v}((ft)_t)(\bar{z},\xi,0)$, we obtain
$$
W(0,1)-W(0,0)-\pi\leq Torsion_1((f_t)_t,\bar{z},\xi)\leq W(0,1)-W(0,0)+\pi.
$$
By Proposition \eqref{main thm torsion}, $W(0,1)-W(0,0)$ is in $(-\pi,0)$, hence
$$
-2\pi<Torsion_1((f_t)_t,\bar{z},\xi)<\pi.
$$
If $\xi$ is in the left half-plane, then $-\xi$ is in the right half-plane and we know that
$$
Torsion_1((f_t)_t,\bar{z},\xi)=Torsion_1((f_t)_t,\bar{z},-\xi)\in(-2\pi,\pi).
$$
\end{proof}


\noindent\textit{Proof of Theorem \ref{thm intro 2}.} The proof of the Theorem is made by induction. The base case, that is the case with $n=1$, is Proposition \ref{main thm torsion}. Concerning the inductive step, assume that the statement holds true for $n\in\N$.\\
\noindent We use the notation of Proposition \ref{PROP IMPO vV}, but we add the dependence on the point $W_{\bar{z}}(s,t)$.\\
\noindent Then $W_{\bar{z}}(0,\cdot)$ is a continuous determination of the angle function $v((f_t)_t)(\bar{z},(0,1),\cdot)$ that satisfies $W_{\bar{z}}(0,1)=\beta\in(-\pi,0)$, that is
$$
W_{f(\bar{z})}(-\pi,0)=-\pi<W_{\bar{z}}(0,1)=W_{f(\bar{z})}(\beta,0)<W_{f(\bar{z})}(0,0)=0.
$$
By Proposition \ref{PROP IMPO vV} we have for any $t$
$$
W_{f(\bar{z})}(-\pi,t)<W_{f(\bar{z})}(\beta,t)=W_{\bar{z}}(0,1+t)<W_{f(\bar{z})}(0,t).
$$
Using $(iii)$ of Proposition \ref{PROP IMPO vV} it holds
$$
W_{f(\bar{z})}(0,t)-\pi<W_{\bar{z}}(0,1+t)<W_{f(\bar{z})}(0,t).
$$
For $t=n$ we have
$$
W_{f(\bar{z})}(0,n)-\pi-W_{f(\bar{z})}(0,0)<W_{\bar{z}}(0,n+1)-W_{\bar{z}}(0,0)<W_{f(\bar{z})}(0,n)-W_{f(\bar{z})}(0,0).
$$
By induction hypothesis, we have
$$
W_{f(\bar{z})}(0,n)-W_{f(\bar{z})}(0,0)\in(-n\pi,0)
$$
and then
$$
W_{\bar{z}}(0,n+1)-W_{\bar{z}}(0,0)\in(-(n+1)\pi,0).
$$
\hfill$\Box$

\noindent Theorem \ref{thm intro 2} implies the following
\begin{corollario}
Let $f:\mathbb{A}\rightarrow\mathbb{A}$ be a positive twist map. Let $\bar{z}\in\mathbb{A}$ be a point at which the torsion exists.\\
\noindent Then
\begin{equation}
Torsion((f_t)_t,\bar{z})\in[-\pi,0]
\end{equation}
where $((f_t)_t)$ is an isotopy joining the identity with $f$.
\end{corollario}
\begin{remark}
The independence of the torsion from the chosen isotopy is assured by Proposition \ref{indipendenza torsione isotopia}.
\end{remark}

\begin{esempio}
Let $f:\mathbb{A}\rightarrow\mathbb{A}$ be a positive twist map. Any point of an Aubry-Mather set has zero torsion. This result has been proved by S.~Crovisier in \cite{crovisier} (see Theorem 1.2).
\end{esempio}

\begin{esempio}
Let $f:\mathbb{A}\rightarrow\mathbb{A}$ be a positive twist map. If $z\in\mathbb{A}$ is a hyperbolic fixed point such that $Df(z)$ has a negative real eigenvalue, then we have $Torsion(f,z)=-\pi$.\\
\noindent To find an example of such a dynamics, consider the fixed point $(0,0)\in\mathbb{A}$ of the standard map $(x,y)\mapsto f_{\lambda}(x,y)=\left(x+y-\frac{\lambda}{2\pi}\sin(2\pi x),y-\frac{\lambda}{2\pi}\sin(2\pi x)\right)$ for $\lambda\geq 4$.
\end{esempio}

\subsection{Properties of Linking number for twist maps}
\noindent The result presented in Corollary \ref{thm intro 3} provides an estimation of the linking number of the orbit of two points. It is an outcome of Theorems \ref{thm intro 1} and \ref{thm intro 2}.\\
\begin{notazione}
On $\R^2$ we fix the counterclockwise orientation and we consider as reference vector field the constant one $X(z)=(0,1)$ for any $z\in\R^2$.
\end{notazione}

\noindent \textit{Proof of Corollary \ref{thm intro 3}.} Let $f:\mathbb{A}\rightarrow\mathbb{A}$ be a positive twist map and let $(f_t)_t$ be an isotopy in $\text{Diff}^1(\mathbb{A})$ joining the identity to $f_1=f$. Let $(F_t)_t$ be the lift in $\text{Diff}^1(\R^2)$ of $(f_t)_t$ such that $F_0=Id_{\R^2}$. So it joins the identity to $F_1=F$, which is a lift of $f$.\\
\noindent Let $z_1,z_2\in\R^2$, $z_1\neq z_2$ and assume that the limit
$$
Linking((F_t)_t,z_1,z_2) = \lim_{n\rightarrow+\infty}Linking_n((F_t)_t,z_1,z_2)
$$
exists.\\
\noindent For any $n\in\N$ denote $l_n$ as the quantity $Linking_n((F_t)_t,z_1,z_1)$. Fix now $n\in\N$. From Corollary \ref{corollario_nacceso_link}, there exists a point $z$ lying on the segment joining $z_1$ and $z_2$, such that
$$
Torsion_n\left((F_t)_t,z,z_2-z_1\right) = l_n.
$$
\noindent Thanks to the trivialization, we have
$$
Torsion_n\left((F_t)_t,z,z_2-z_1\right) = Torsion_n\left((f_t)_t,\bar{z},z_2-z_1\right)
$$
where $\bar{z}\in\mathbb{A}$ is the projection on the annulus of the point $z\in\R^2$. Therefore Theorem \ref{thm intro 2} tells us that
\begin{equation}\label{torsione nella dim di ln}
Torsion_n((F_t)_t,z,(0,1))\in(-\pi,0).
\end{equation}
By Lemma \ref{lemma vettore n acceso}, it holds
$$
\abs{Torsion_n\left((F_t)_t,z,z_2-z_1\right)-Torsion_n((F_t)_t,z,(0,1))}<\dfrac{\pi}{n}
$$
and then, by \eqref{torsione nella dim di ln},
$$
l_n=Torsion_n\left((F_t)_t,z,z_2-z_1\right)\in\left( -\pi-\dfrac{\pi}{n},\dfrac{\pi}{n} \right).
$$
\noindent We deduce that
$$
l_n = Linking_n((F_t)_t,z_1,z_2)\in\left(-\pi-\dfrac{\pi}{n},\dfrac{\pi}{n}\right).
$$
\noindent Since this holds for any fixed $n\in\N, n\neq 0$, passing to the limit, we conclude that
$$
Linking((F_t)_t,z_1,z_2)\in[-\pi,0].
$$
\hfill$\Box$

\indent We also give an estimation of finite-time linking number under some further assumptions.
\begin{proposizione}\label{proprieta linking n}
Let $F:\R^2\rightarrow\R^2$ be a lift of a positive twist map and let $p_1:\R^2\rightarrow\R$ be the projection over the first coordinate. Let $z_1,z_2\in\R^2,z_1\neq z_2$ be such that $p_1(z_1)=p_1(z_2)$. Then for any $n\in\N$
\begin{equation}
Linking_n((F_t)_t,z_1,z_2)\in(-\pi,0).
\end{equation}
\end{proposizione}
\begin{proof}
Arguing by contradiction, assume that there exist points $z_1,z_2\in\R^2,z_1\neq z_2$ with $p_1(z_1)=p_1(z_2)$ and $n\in\N, n\neq 0$ such that
\begin{equation}
Linking_n((F_t)_t,z_1,z_2)=l
\end{equation}
with either $l$ smaller or equal $-\pi$ or $l$ greater or equal $0$. From the condition over the first coordinate projection, the vector $z_2-z_1$ joining the two points is vertical. By Corollary \ref{corollario_nacceso_link}, there exists a point $z\in\R^2$ lying on the segment joining $z_2$ and $z_1$ such that
\begin{equation}
Torsion_n\left((F_t)_t,z,z_2-z_1\right)=Torsion_n((F_t)_t,z,(0,1))=l.
\end{equation}
The value $l$ does not belong to the interval $(-\pi,0)$. This contradicts Theorem \ref{thm intro 2} and we conclude.
\end{proof}

\noindent Remark that if two points $z_1,z_2$ do not have the same first coordinate projection, then the result of Proposition \ref{proprieta linking n} does not hold, as shown by the following two examples. Moreover, Examples \ref{exempio link 1} and \ref{exempio link 2} show us that the extremal values $0$ and $-\pi$ of the admissible interval for the linking number in Corollary \ref{thm intro 3} can be actually attained.
\begin{esempio}\label{exempio link 1}
Consider a lift of a $\mathcal{C}^1$ diffeomorphism on $\mathbb{A}$ (not only lifts of twist maps): the linking number of any two points $z_1,z_2=z_1+(2\pi,0)$ is null.
\end{esempio}

\begin{esempio}\label{exempio link 2}
Let $F:\R^2\rightarrow\R^2$ be a lift of a positive twist map on $\mathbb{A}$. Assume that $z_0$ is a hyperbolic fixed point such that $DF(z_0)$ has a negative real eigenvalue of modulus strictly smaller than $1$. Let $z_1$ be a point lying on one of the stable branches of $z_0$. Then $Linking(F,z_0,z_1)=-\pi$.
\end{esempio}

\subsection{Crovisier's Torsion for twist maps: definition and comparison}\label{crovisier}
\indent In \cite{crovisier} S.~Crovisier gives another definition of torsion for a positive twist map. It seems natural then comparing the two definitions: we prove that the two definitions are equivalent and so we deduce that Crovisier's results hold also refering to our Definition \ref{definizione_torsion}.\\
\noindent As before, we fix the counterclockwise orientation and we consider as  reference vector field on $\mathbb{A}$ the constant one $(0,1)$.
\begin{definizione}[Crovisier's definition in \cite{crovisier}]
Let $f:\mathbb{A}\rightarrow\mathbb{A}$ be a positive twist map. Let $\bar{z}\in\mathbb{A}$ and let $\xi\in T_{\bar{z}}\mathbb{A}\setminus\{0\}$. Denote as $z\in\R^2$ a lift of the point $\bar{z}$. We define the function
$$
\theta^0:T_{\bar{z}}\mathbb{A}\setminus\{0\}\rightarrow(-2\pi,0]\subset\R
$$
$$
\xi\mapsto\theta^0(\xi)
$$
as the measure of the oriented angle between the vertical vector $(0,1)$ and $\xi$ contained in the interval $(-2\pi,0]$. The quantity $\theta^0(Df(\bar{z})(0,1))$ is then the measure of the oriented angle between $(0,1)$ and $Df(\bar{z})(0,1)$ contained in the interval $(-2\pi,0]$.\\
\noindent We define the function
$$
\theta^1:T_{\bar{z}}\mathbb{A}\setminus\{0\}\rightarrow(\theta^0(Df(\bar{z})(0,1))-2\pi,\theta^0(Df(\bar{z})(0,1))]
$$
$$
\xi\mapsto\theta^1(\xi)
$$
as the measure of the oriented angle between $(0,1)$ and $Df(\bar{z})\xi$ contained in the real interval $(\theta^0(Df(\bar{z})(0,1))-2\pi,\theta^0(Df(\bar{z})(0,1))]$. We define the following function 
$$
\theta:T_{\bar{z}}\mathbb{A}\setminus\{0\}\rightarrow\R
$$
$$
\xi\mapsto\theta(\xi) := \theta^1(\xi)-\theta^0(\xi)
$$
which is a measure of the oriented angle between $\xi$ and $Df(\bar{z})\xi$.\\
\noindent For a given $n\in\Z$ define
\begin{equation}\label{torsione discreta crovisier}
\theta_n(\xi) := \begin{cases}
\sum_{0\leq k\leq n-1}\theta (Df^k(\bar{z})\xi)\qquad n\geq 0 \\
-\theta_{-n}(Df^n(\bar{z})\xi)\qquad n<0.
\end{cases}
\end{equation}
\end{definizione}

\noindent Observe that for $k\in\N$, the quantity $\theta(Df^k(\bar{z})\xi)$ is the difference between $\theta^1(Df^k(\bar{z})\xi)$ and $\theta^0(Df^k(\bar{z})\xi)$, where $\theta^0(Df^k(\bar{z})\xi)$ is the measure, contained in $(-2\pi,0]$, of the oriented angle between the vectors $(0,1)$ and $Df^k(\bar{z})\xi$. These vectors lie in the tangent space $T_{f^k(\bar{z})}\mathbb{A}$. On the other hand, $\theta^1(Df^k(\bar{z})\xi)$ is the measure, contained in the interval $(\theta^0(Df(f^k(\bar{z}))(0,1))-2\pi,\theta^0(Df(f^k(\bar{z}))(0,1))]$, of the oriented angle between $(0,1)$ and $Df^{k+1}(\bar{z})\xi$. These vectors lie in the tangent space $T_{f^{k+1}(\bar{z})}\mathbb{A}$.

\begin{definizione}
An invariant set $\bar{X}\subset\mathbb{A}$ with respect to a positive twist map $f:\mathbb{A}\rightarrow\mathbb{A}$ and its lift $X\subset\R^2$ have null torsion if for any $\bar{z}\in\bar{X}$ and any $\xi\in T_{\bar{z}}\mathbb{A}\setminus\{0\}=T_z\R^2\setminus\{0\}$ it holds
$$
\lim_{n\rightarrow\pm\infty}\dfrac{\theta_n(\xi)}{n}=0.
$$
\end{definizione}

For a given $n\in\N$, a fixed point $\bar{z}\in\mathbb{A}$ and a fixed vector $\xi\in T_{\bar{z}}\mathbb{A}\setminus\{0\}$ the quantities $\theta_n(\xi)$ and $nTorsion_n\left((f_t)_t,\bar{z},\xi\right)$, as defined in \eqref{intro torsion tempi discreti}, coincide. Hence, Crovisier's and B\'eguin's definitions of torsion are the same.

\begin{proposizione}
Let $\bar{z}\in\mathbb{A}$ and $\xi\in T_{\bar{z}}\mathbb{A}\setminus\{0\}$. Let $f:\mathbb{A}\rightarrow\mathbb{A}$ be a positive twist map and let $(f_t)_t$ be an isotopy in $\text{Diff}^1(\mathbb{A})$ joining the identity to $f_1=f$. Then (see Definitions \eqref{torsione discreta crovisier} and \eqref{intro torsion tempi discreti})
\begin{equation}
nTorsion_n\left((f_t)_t,\bar{z},\xi\right) = \theta_n(\xi).
\end{equation}
\end{proposizione}

\begin{proof}
Crovisier's quantity $\theta_n(\xi)$, as defined in \eqref{torsione discreta crovisier}, is then
$$
\theta_n(\xi) = \sum_{0\leq k\leq n-1}\theta(Df^k(z)\xi) = \sum_{0\leq k\leq n-1}\left( \theta^1(Df^k(z)\xi) - \theta^0(Df^k(z)\xi) \right).
$$
\noindent On the other hand, B\'eguin's quantity, as presented in \eqref{intro torsion tempi discreti}, is
$$
n Torsion_n\left((f_t)_t,z,\xi\right) = \sum_{0\leq k\leq n-1} Torsion_1\left((f_t)_t,f^k(z),Df^k(z)\xi\right)=
$$
$$
= \sum_{0\leq k\leq n-1}\left( \tilde{v}\left((f_t)_t\right)\left(f^k(z),Df^k(z)\xi,1\right) - \tilde{v}\left((f_t)_t\right)\left(f^k(z),Df^k(z)\xi,0\right) \right).
$$

We prove that for any $0\leq k\leq n-1$
$$
\theta^1(Df^k(z)\xi)-\theta^0(Df^k(z)\xi) = \tilde{v}\left((f_t)_t\right)\left(f^k(z),Df^k(z)\xi,1\right)-\tilde{v}\left((f_t)_t\right)\left(Df^k(z),Df^k(z)\xi,0\right)
$$
and this concludes the proof.\\
\noindent Show it for $k=0$, being the proof of the equality of the other terms the same.\\
\noindent The oriented angles involved are the same. Indeed, $\theta^0(\xi)$ is a measure of the oriented angle between $(0,1)$ and $\xi$, that is the angle $v((f_t)_t)(z,\xi,0)$; $\theta^1(\xi)$ is a measure of the oriented angle between $(0,1)$ and $Df(z)\xi$, that is the angle $v((f_t)_t)(z,\xi,1)$.\\
\noindent The quantity $\tilde{v}((f_t)_t)(z,\xi,1)-\tilde{v}((f_t)_t)(z,\xi,0)$ does not depend on the chosen lift. We show that, by choosing the lift so that $\tilde{v}((f_t)_t)(z,\xi,0)=\theta^0(\xi)$, it holds $\tilde{v}((f_t)_t)(z,\xi,1)=\theta^1(\xi)$. This implies the required equality.\\
\noindent We refer to the notation of Proposition \ref{PROP IMPO vV}. We choose the lift so that $\tilde{v}((f_t)_t)(z,\xi,0)=\theta^0(\xi)\in(-2\pi,0]$. There exists $s\in(-2\pi,0]$ such that $W_z(s,\cdot)=\tilde{v}((f_t)_t)(z,\xi,\cdot)$.\\
\noindent Observe that $W_z(-2\pi,0)<W_z(s,0)\leq W_z(0,0)$. By point $(ii)$ of Proposition \ref{PROP IMPO vV} it holds
$$
W_z(-2\pi,1)<W_z(s,1)=\tilde{v}((f_t)_t)(z,\xi,1)\leq W_z(0,1).
$$
Point $(iii)$ of Proposition \ref{PROP IMPO vV} tells us that $W_z(-2\pi,1)=W_z(0,1)-2\pi$, so we have
$$
W_z(0,1)-2\pi<\tilde{v}((f_t)_t)(z,\xi,1)\leq W_z(0,1).
$$
By Theorem \ref{thm intro 2} for $n=1$ it holds $W_z(0,1)\in(-\pi,0)\subset(-2\pi,0]$. Being lifts of the same angle both in $(-2\pi,0]$, we have $W_z(0,1)=\theta^0(Df(z)(0,1))$.\\
\noindent We then conclude that $W_z(s,1)=\tilde{v}((f_t)_t)(z,\xi,1)\in(-2\pi+\theta^0(Df(z)(0,1)),\theta^0(Df(z)(0,1))]$ and so $\tilde{v}((f_t)_t)(z,\xi,1)=\theta^1(\xi)$ being lifts of the same angle both contained in the interval\\ \noindent$(-2\pi+\theta^0(Df(z)(0,1)),\theta^0(Df(z)(0,1))]$.
\end{proof}

\indent We recall the result obtained by S.~Crovisier in \cite{crovisier}. Since the two definitions of torsion are equivalent, this result holds true also refering to the torsion presented in Definition \ref{definizione_torsion}. For the definition of \emph{well-ordered sets} we refer to \cite{chenciner} and \cite{crovisier}.
\begin{definizione}[Well-ordered set]
A set $\bar{E}\subset\mathbb{A}$, not empty and invariant for $f$, and its lift $E\subset\R^2$ are said \emph{well-ordered} if
\begin{itemize}
\item[$(i)$] $\bar{p}_1:\bar{E}\rightarrow\mathbb{T}$ is injective;
\item[$(ii)$] for any $z,z'\in E$, lifts of points $\bar{z},\bar{z}'\in\bar{E}$, such that $p_1(z)<p_1(z')$, it holds that $p_1(F(z))<p_1(F(z'))$.
\end{itemize}
\end{definizione}
A rotation number is associated to any well-ordered set (see \cite{lecalvezproprietes}).\\
\begin{teorema}[Theorem 1.2 in \cite{crovisier}]
Let $f:\mathbb{A}\rightarrow\mathbb{A}$ be a positive twist map and let $(f_t)_t$ be an isotopy joining the identity to $f$. Then, any well-ordered set with irrational rotation number has null torsion.
\end{teorema}


\appendix
\section{Proof of Proposition \ref{stimaangoli}}\label{appendice prova prop}

We now present the proof of the technical Proposition \ref{stimaangoli}, used in the discussion of case $(i)$ of Theorem \ref{thm intro 1} (see Subsection \ref{sezione prova linking-torsion}). Consider $J_1(s),J_2(s)$ for $s\in I$. By hypothesis $\pi_M\circ J_1=\pi_M\circ J_2$, so they lie on the same tangent space.\\
\noindent Four different cases can occur:
\begin{itemize}
\item[$(1)$]$J_1(s),J_2(s)$ are positively colinear, i.e. $J_1(s)=\lambda J_2(s)$ for some $\lambda>0$. Hence the associated angle function satisfies $\bar{\theta}(s)=0\ mod\,2\pi$ and any continuous determination $\theta$ verifies $\theta(s)=2\pi k,k\in\Z$.
\item[$(2)$]$J_1(s),J_2(s)$ are negatively colinear, i.e. $J_1(s)=\lambda J_2(s)$ for some $\lambda <0$. Hence the associated angle function satisfies $\bar{\theta}(s)=\pi\ mod\,2\pi$ and any continuous determination $\theta$ verifies $\theta(s)=\pi+2\pi k, k\in\Z$.
\item[$(3)$]$J_1(s),J_2(s)$ are linearly independent and $(J_1(s),J_2(s))$ is a direct basis. Therefore the associated angle function satisfies $\bar{\theta}(s)\in(0,\pi)\ mod\,2\pi$ and any continuous determination $\theta$ verifies $\theta(s)\in(2\pi k, \pi+2\pi k), k\in\Z$.
\item[$(4)$]$J_1(s),J_2(s)$ are linearly independent and $(J_1(s),J_2(s))$ is a non-direct basis. Therefore the associated angle function satisfies $\bar{\theta}(s)\in(\pi,2\pi)\ mod\,2\pi$ and any continuous determination $\theta$ verifies $\theta(s)\in(\pi+2\pi k, 2\pi(k+1)),k\in\Z$.
\end{itemize}

We denote as $\bar{\Theta}(s)$ the oriented angle between $Df\circ J_1(s)$ and $Df\circ J_2(s)$.
\begin{lemma}\label{il pi non lo prendo come valore}
Let $I\subset\R$ and let $M,N$ be 2-dimensional oriented Riemannian manifolds. Let $f:M\rightarrow N$ be a local diffeomorphism which preserves the orientation and let $J_1,J_2:I\rightarrow TM$ be continuous functions that never vanish. Assume also that $\pi_M\circ J_1=\pi_M\circ J_2$. Let $\bar{\theta},\bar{\Theta}:I\rightarrow\mathbb{T}$ be the oriented angles, respectively, between the image vectors $J_1,J_2$ and the image vectors $Df\circ J_1,Df\circ J_2$.\\
\noindent Then, for any $s\in I$
\begin{equation}
(\bar{\theta}-\bar{\Theta})(s)\neq\pi\quad mod\,2\pi.
\end{equation}
\end{lemma}
\noindent We postopone the proof of this lemma.\\
\noindent Let $\theta$ be a chosen continuous determination of the angle $\bar{\theta}$. Let fix $s_0\in I$. Depending on the cases, we have
$$
\theta(s_0)\begin{cases}
=2k\pi\qquad\text{ if }\bar{\theta}(s_0)=0\ mod\,2\pi \\
=2k\pi+\pi\qquad\text{ if }\bar{\theta}(s_0)=\pi\ mod\,2\pi \\
\in (0,\pi)+2k\pi\qquad\text{ if }\bar{\theta}(s_0)\in(0,\pi)\ mod\,2\pi \\
\in (\pi,2\pi)+2k\pi\qquad\text{ if }\bar{\theta}(s_0)\in(\pi,2\pi)\ mod\,2\pi
\end{cases}
$$
where $k\in\Z$.

\noindent Choose a measure $\Theta(s_0)$ of the angle $\bar{\Theta}(s_0)$ such that
$$
\abs{\theta(s_0)-\Theta(s_0)}<\pi
$$

\noindent By the continuity of the chosen determination $\Theta$, from the relation holding in $s_0$ just shown and from Lemma \ref{il pi non lo prendo come valore}, for any $s\in I$ we conclude
$$
\abs{\theta(s)-\Theta(s)}<\pi.
$$

~\newline
\textit{Proof of Lemma \ref{il pi non lo prendo come valore}.} 
As remarked above, only four cases can occur concerning the relative positions of vectors $J_1(s),J_2(s)$ for any fixed $s\in I$.\\
\noindent We then show that for any $s$
$$
\bar{\theta}(s)-\bar{\Theta}(s)\neq\pi\ mod\,2\pi.
$$
Arguing by contradiction, assume that there exists $s$ so that $\bar{\theta}(s)-\bar{\Theta}(s)=\pi\ mod\,2\pi$. Then the couples of vectors $(J_1(s),J_2(s))$ and $(Df\circ J_1(s),Df\circ J_2(s))$ belong to different cases and this is a contradiction. Indeed, since $f$ is a local diffeomorphism which preserves the orientation, looking at the relative position of vectors $Df\circ J_1(s), Df\circ J_2(s)$, the same four cases presented above can occur and for any fixed $s\in I$ we remain in the same case as $J_1(s),J_2(s)$.
\hfill$\Box$

\clearpage
\bibliographystyle{siam}
\bibliography{Bibliography}

\begin{thebibliography}{10}

\bibitem{angenent}
{\sc S.~B. Angenent}, {\em The periodic orbits of an area preserving twist
  map}, Communications in Mathematical Physics, 115 (1988), pp.~353--374.

\bibitem{beguin}
{\sc F.~B\'eguin and Z.~R. Boubaker}, {\em Existence of orbits with non-zero
  torsion for certain types of surface diffeomorphisms}, Journal of the
  Mathematical Society of Japan, 65 (2013), pp.~137--168.

\bibitem{chenciner}
{\sc A.~Chenciner}, {\em La dynamique au voisinage d'un point fixe elliptique
  conservatif: de {P}oincar{\'e} et {B}irkhoff {\'a} {A}ubry et {M}ather},
  Ast\'erisque,  (1985), pp.~147--170.
\newblock Seminar Bourbaki, Vol. 1983/84.

\bibitem{crovisier}
{\sc S.~Crovisier}, {\em Ensembles de torsion nulle des applications d\'eviant
  la verticale}, Bulletin de la Soci\'et\'e math\'ematique de France, 131
  (2003), pp.~23--39.

\bibitem{docarmo}
{\sc M.~P. Do~Carmo}, {\em Differential geometry of curves and surfaces},
  Prentice-Hall, Inc., 1976.

\bibitem{katok}
{\sc A.~Katok and B.~Hasselblatt}, {\em Introduction to the Modern Theory of
  Dynamical Systems}, Cambridge University Press, 1995.

\bibitem{lecalvezproprietes}
{\sc P.~Le~Calvez}, {\em Propri{\'e}t{\'e}s dynamiques des diff{\'e}omorphismes
  de l'anneau et du tore}, Ast\'erisque,  (1991), p.~131.

\bibitem{matherexistence}
{\sc J.~N. Mather}, {\em Existence of quasiperiodic orbits for twist
  homeomorphisms of the annulus}, Topology. An International Journal of
  Mathematics, 21 (1982), pp.~457--467.

\bibitem{matherglancing}
\leavevmode\vrule height 2pt depth -1.6pt width 23pt, {\em Glancing billiards},
  Ergodic theory and dynamical systems, 2 (1982), pp.~397--403.

\bibitem{matheramount}
\leavevmode\vrule height 2pt depth -1.6pt width 23pt, {\em Amount of rotation
  about a point and the {M}orse index}, Communications in Mathematical Physics,
  94 (1984), pp.~141--153.

\bibitem{mathervariational}
\leavevmode\vrule height 2pt depth -1.6pt width 23pt, {\em Variational
  construction of orbits of twist diffeomorphisms}, Journal of the American
  Mathematical Society, 4 (1991), pp.~207--263.

\bibitem{matsumoto}
{\sc S.~Matsumoto and H.~Nakayama}, {\em On the {R}uelle invariants for
  diffeomorphisms of the two torus}, Ergodic Theory and Dynamical Systems, 22
  (2002), pp.~1263--1267.

\bibitem{moser}
{\sc J.~Moser}, {\em Monotone twist mappings and the calculus of variations},
  Ergodic theory and dynamical systems, 6 (1986), pp.~401--413.

\bibitem{ruelle}
{\sc D.~Ruelle}, {\em Rotation numbers for diffeomorphisms and flows}, Annales
  de l'Institute Henri Poincar\'e. Physique th\'eorique, 42 (1985),
  pp.~109--115.

\end{thebibliography}

\end{document}